\documentclass[11pt]{article}

\font\notsosmall=cmr8

\usepackage{lipsum} 

\usepackage{amssymb}
\usepackage{amsthm}
\usepackage{amsfonts}
\usepackage[latin2]{inputenc}
\usepackage{amsmath}
\usepackage{anysize}
\usepackage{hyperref}
\usepackage{bbm}

\usepackage{comment}
\usepackage{enumitem}

\newtheorem{theorem}{Theorem}[section]
\newtheorem{lemma}[theorem]{Lemma}
\newtheorem{prop}[theorem]{Proposition}
\newtheorem{proposition}[theorem]{Proposition}
\newtheorem{corollary}[theorem]{Corollary}
\newtheorem{conjecture}[theorem]{{Conjecture}}

\newtheorem{claim}[theorem]{{Claim}}

\theoremstyle{remark}
\newtheorem{remark}[theorem]{Remark}

\theoremstyle{definition}
\newtheorem{example}[theorem]{{Example}}

\newtheorem{definition}[theorem]{{Definition}}
\newtheorem{hypothesis}[theorem]{{Hypothesis}}
\newtheorem{notation}[theorem]{{Notation}}

\def\bclaim{\begin{claim}}
\def\eclaim{\end{claim}}
\def\bdefin{\begin{definition}}
\def\edefin{\end{definition}}
\def\bcor{\begin{corollary}}
\def\ecor{\end{corollary}}
\def\bthm{\begin{theorem}}
\def\ethm{\end{theorem}}
\def\bconj{\begin{conjecture}}
\def\econj{\end{conjecture}}
\def\blem{\begin{lemma}}
\def\elem{\end{lemma}}
\def\blemma{\begin{lemma}}
\def\elemma{\end{lemma}}
\def\bprop{\begin{prop}}
\def\eprop{\end{prop}}
\def\bremark{\begin{remark}}
\def\eremark{\end{remark}}
\def\bhyp{\begin{hypothesis}}
\def\ehyp{\end{hypothesis}}
\def\bnot{\begin{notation}}
\def\enot{\end{notation}}
\def\bexample{\begin{example}}
\def\eexample{\end{example}}

\def\MA{Monge--Amp\`ere }
\def\K{K\"ahler }
\def\i{\sqrt{-1}}
\def\del{\partial}
\def\dbar{\bar\partial}
\def\ddbar{\del\dbar}
\def\ra{\rightarrow}
\def\eps{\epsilon}
\newcommand{\RR}{\mathbb{R}}
\newcommand{\CC}{\mathbb{C}}
\newcommand{\NN}{\mathbb{N}}

\newcommand{\PP}{\mathbb{P}}
\def\del{\partial}
\newcommand{\calH}{\mathcal{H}}

\DeclareMathOperator{\Ric}{Ric}

\def\w{\wedge}
\def\o{\omega}

\newcommand{\ecal}{\mathcal{E}}
\newcommand{\hcal}{\mathcal{H}}

\newcommand{\lcal}{\mathcal{L}}\newcommand{\mcal}{\mathcal{M}}

\newcommand{\rcal}{\mathcal{R}}

  \def\calR{\rcal}
\def\calM{\mcal}  \def\calL{\lcal}
 \def\calH{\hcal} \def\H{\hcal}
  \def\calE{\ecal}

\def\a{\alpha} \def\be{\beta}
\def\o{\omega} 
 \def\Th{\Theta}
\def\vp{\varphi} \def\eps{\epsilon}

\def\K{K\"ahler } 
\def\KE{K\"ahler--Einstein } \def\KEno{K\"ahler--Einstein}
\def\KR{K\"ahler--Ricci }
\def\KRS{K\"ahler--Ricci soliton }

\def\Ric{\hbox{\rm Ric}\,} 

\def\ovpn{\o^n_{\vp}}

\def\h#1{\hbox{#1}}
\def\text{\textstyle}

\def\q{\quad} \def\qq{\qquad}

\def\PSH{\mathrm{PSH}}

\def\Aut{{\operatorname{Aut}}}

\def\Isom{{\operatorname{Isom}}}
\def\Cinf{C^\infty}

\def\h#1{\hbox{#1}}
\def\MA{Monge--Amp\`ere }

\def\bpf{\begin{proof}}
\def\epf{\end{proof}}
\def\beq{\begin{equation}}
\def\eeq{\end{equation}}
\def\beqno{\begin{equation*}}
\def\eeqno{\end{equation*}}
\def\eaeq{\end{aligned}}
\def\baeq{\begin{aligned}}

\newcommand\blfootnote[1]{%
  \begingroup
  \renewcommand\thefootnote{}\footnote{#1}%
  \addtocounter{footnote}{-1}%
  \endgroup
}

\def\Id{{\operatorname{Id}}}

\def\Ric{\hbox{\rm Ric}\,}

\def\ovp{\omega_{\varphi}}

\def\ovpn{\omega^n_{\varphi}}
\def\etan{\eta^n}
\def\on{\omega^n}

\def\Ho{\calH_\omega}

\def\fk{\mathfrak{k}}
\def\fz{\mathfrak{z}}
\def\tfk{\tilde{\mathfrak{k}}}
\def\fg{\mathfrak{g}}
\def\fa{\mathfrak{a}}
\def\fs{\mathfrak{s}}
\def\op{\oplus}

\def\lb{\label}

\def\Ent{\h{\rm Ent}}

\def\aut{\h{\rm aut}}
\def\AutMJ{\Aut(M,\JJJ)}
\def\AutMJz{\Aut(M,\JJJ)_0}
\def\autMJ{\aut(M,\JJJ)}

\def\E{\calE}

\def\vpt{\vp_t}
\def\ovpt{\o_{\vp_t}}

\def\beq{\begin{equation}}
\def\eeq{\end{equation}}
\def\beqno{\begin{equation*}}
\def\eeqno{\end{equation*}}
\def\eaeq{\end{aligned}}
\def\baeq{\begin{aligned}}

\def\bpf{\begin{proof}}
\def\epf{\end{proof}}

\def\a{\alpha} \def\be{\beta}
\def\o{\omega} 
 \def\Th{\Theta}
\def\vp{\varphi} \def\eps{\epsilon}

\def\ovp{{\o_{\vp}}}
\def\on{\omega^n}
\def\ovpn{\omega_{\vp}^n}

\def\Ric{\hbox{\rm Ric}\,}

\def\ovpn{\o^n_{\vp}}

\def\h#1{\hbox{#1}}

\def\strutdepth{\dp\strutbox}
\def\specialstar{\vtop to \strutdepth{
    \baselineskip\strutdepth
    \vss\llap{$\star$\ \ \ \ \ \ \ \ \  }\null}}
\def\marginalstar{\strut\vadjust{\kern-\strutdepth\specialstar}}

\def\marginal#1{\strut\vadjust{\kern-\strutdepth
    {\vtop to \strutdepth{
    \baselineskip\strutdepth
    \vss\llap{{ \small #1 }}\null}
    }}
    }

\def\etan{\eta^n}
\def\Cinf{C^\infty}
\def\CinfM{\Cinf(M)}
\def\JJJ{\h{\rm J}}
\def\JJJsml{\h{\notsosmall J}}
\def\isomsml{\h{\notsosmall isom}}
\def\isom{\h{\rm isom}}
\def\V{V^{-1}}
\def\fovpt{f_{\ovpt}}

\title{Tian's properness conjectures and Finsler geometry of the space of
K\"ahler metrics}

\author{Tam\'as Darvas and Yanir A. Rubinstein
\blfootnote{2010 Mathematics subject classification 53C55, 32W20, 32U05.}
}
\date{\vspace{-0.3in}}

\begin{document}

\maketitle
\begin{abstract}
Well-known conjectures of Tian predict that 
existence of canonical K\"ahler metrics should  be equivalent to 
various notions of properness of 
Mabuchi's K-energy functional.
In some instances this has been verified, especially under restrictive
assumptions on the automorphism group.
We provide counterexamples to the original conjecture in the presence
of continuous automorphisms. The construction hinges upon 
an alternative approach to properness that uses in an essential
way the metric completion with respect to a {\it Finsler } metric
and its quotients with respect to group actions.
This approach also allows us to formulate and prove 
new optimal versions of Tian's conjecture in the 
setting of smooth and singular K\"ahler--Einstein metrics,
with or without automorphisms, as well as for
K\"ahler--Ricci solitons. Moreover, we reduce both 
Tian's original conjecture (in the absence of automorphisms) 
and our modification of it
(in the presence of automorphisms) in the general case of constant scalar curvature
metrics to a conjecture on 
regularity of minimizers
of the K-energy in the Finsler metric completion.
Finally, our results also resolve Tian's conjecture on the
Moser--Trudinger inequality for Fano manifolds with
\KE metrics. 
\end{abstract}


\section{Introduction}

The main motivation for our work is Tian's properness conjecture.
Consider the space 
\begin{equation}
\label{HEq}
\textstyle\calH
=
\{\omega_\vp:=\o+\i\ddbar\vp \,:\, \vp\in C^{\infty}(M), \,  \omega_\vp>0\}
\end{equation}
of all \K metrics representing a fixed cohomology class  
on a compact \K manifold $(M,\JJJ,\o)$ 

Motivated by results in conformal geometry 
and the direct method in the calculus of variations,
in the 90's Tian  introduced the notion of ``properness on $\calH$"
\cite[Definition 5.1]{T1994}
in terms of the Aubin nonlinear 
energy functional $J$ \cite{Aubin84}  and the Mabuchi K-energy $E$ 
\cite{Mabuchi87} as follows.

\bdefin
The functional $E:\calH\ra\RR$ is said to be proper if 
\beq
\label{PropernessEq}
\forall\,\o_j\in\calH,\q
\lim_jJ(\o_j)\ra \infty 
\q\Longrightarrow\q
\lim_jE(\o_j)\ra \infty. 
\eeq
\edefin
\noindent
Tian made the following influential conjecture 
\cite[Remark 5.2]{T1994},
\cite[Conjecture 7.12]{Tianbook}.
Denote by $\AutMJz$ the identity component of the group of 
automorphisms of $(M,\JJJ)$, and denote by
$\aut(M,\JJJ)$ its Lie algebra, consisting of holomorphic vector fields.

\bconj
\lb{TianConj}
Let $(M,\JJJ,\o)$ be a compact \K manifold. 

\noindent
(i) If $\aut(M,\JJJ)=0$ then $\H$ contains a constant scalar curvature metric
if and only if $E$ is proper.

\smallskip
\noindent
(ii) Let $K$ be a maximally compact subgroup of $\AutMJz$. Then 
$\H$ contains a constant scalar curvature metric
if and only if $E$ is proper on the subset $\H^K\subset\H$
consisting of $K$-invariant metrics.
\econj

Tian's conjecture is central in \K geometry and has attracted
much work over the past two decades including motivating 
much work on equivalence between 
algebro-geometric notions of stability and existence of 
canonical metrics, as well as on the interface
of pluripotential theory and \MA equations. We refer to the surveys \cite{Thomas,PhongSturmSurvey,Tian2012,PhongSongSturm,R14}.

We provide counterexamples to Conjecture \ref{TianConj} (ii)---see Example \ref{MainExam} below. 
Perhaps more interesting than the examples themselves is the 
realization
that Tian's conjecture should be modified and phrased in terms of a Finsler structure on the space of \K metrics
and properties of its metric completion.
The metric completion approach turns out not only to be convenient but
indispensable. In fact, our results show that properness with respect
to the Finsler distance function {\it characterizes} existence of
canonical \K metrics in many cases. 
It allows us to simultaneously unify, extend, and give 
new proofs of a number of instances of Tian's conjecture and our modification
of it, as well as resolve some of the remaining open cases.
Moreover, we reduce the properness conjecture 
in the most general case of constant scalar curvature metrics
to a problem on minimizers of the K-energy in the Finsler
metric completion. 
Finally, our results 
immediately imply Tian's conjecture on the
Moser--Trudinger inequality for \KE Fano manifolds.
This is the \K geometry analogue of Aubin's strong Moser--Trudinger
inequality on $S^2$ in conformal geometry.

Approaching problems in \K geometry through 
an infinite-dimensional
{\it Riemannian} perspective goes back to Calabi in 1953 \cite{Calabi54} and 
later Mabuchi in 1986 \cite{Mabuchi87}. These works proposed two different 
weak Riemannian metrics of $L^2$ type which have been studied extensively
since. 
Historically, Calabi raised the question of computing the completion
of his metric, which suggested  a relation
between the existence of canonical metrics
on the finite-dimensional manifold $M$ 
and the metric 
completion of the infinite-dimensional space $\H$.
The first result in this spirit is due to
Clarke--Rubinstein \cite{ClarkeR}
who computed the Calabi metric completion, 
and proved the existence of K\"ahler--Einstein metrics on $M$ is equivalent
to the Ricci flow converging in the Calabi metric completion.
Confirming a conjecture of Guedj \cite{guedj}, 
the Mabuchi metric completion was computed recently in
\cite{da3} and in  \cite{da4} a corresponding result for the Ricci flow in the Mabuchi
metric completion was proved.
Other results include the work of Streets \cite{St}, who shows that one gains new insight on the long time behavior of the Calabi flow by placing it in the context of the abstract Riemannian metric completion of the Mabuchi metric. 
In Darvas--He \cite{dh}, the asymptotic behavior of the K\"ahler-Ricci flow in 
the metric completion is related to destabilizing geodesic rays.
We refer the reader to the survey \cite{R14} for more references. 

Perhaps surprisingly, 
a key observation of the present article is 
that not a Riemannian but rather a {\it Finsler} 
metric encodes the asymptotic behavior of essentially all 
energy functionals on $\calH$
whose critical points are precisely various types of canonical
metrics in \K geometry. 
In fact, as pointed out in Remark \ref{Dingd1Remark}, the same
kind of statement is in general false for the much-studied
Riemannian metrics of Calabi and Mabuchi.
The Finsler structure that we use was introduced in 
 \cite{da4} where its metric completion was computed.

\section{Results}

Much of the progress on Conjecture \ref{TianConj} has focused on 
the case of \KEno, \KE edge, or  \KR soliton metrics. 
In the setting of \KE metrics,
one direction of the conjecture (properness implies existence)
follows from work of Ding--Tian \cite{dt},
while the converse for Conjecture \ref{TianConj} (i) was established
by Tian \cite{Tian97} under a technical assumption that
was removed by Tian--Zhu \cite{TZ00}.
The result furnished the first 
`stability' criterion
equivalent to the existence of \KE metrics, in the
absence of holomorphic vector fields.

Our first result disproves Conjecture \ref{TianConj} (ii) 
already in the setting of \KE metrics, and establishes an optimal 
replacement for it. 
\begin{theorem}
\label{KEGexistenceIntroThm}
Suppose $(M,\JJJ,\o)$ 
is Fano and that $K$ is a maximal compact subgroup of $\AutMJz$ with $\o\in\H^K$.
The following are equivalent:

\smallskip
\noindent (i) 
There exists a K\"ahler--Einstein metric in $\mathcal H^K$
and $\AutMJz$ has finite center.

\noindent (ii) 
There exists $C,D>0$ such that $E(\eta) \geq CJ(\eta)-D, \ \eta \in \mathcal H^K$.  \end{theorem}

The estimate in (ii) gives a concrete version of the properness condition 
\eqref{PropernessEq}. The direction (i) $\,\Rightarrow\,$ (ii) is 
due to Phong et al. \cite[Theorem 2]{pssw}, 
building on earlier work of Tian \cite{Tian97} and Tian--Zhu \cite{TZ00}
in the case $\aut(M,\JJJ)=0$,
who obtained a weaker inequality in (ii) with $J$ replaced by $J^\delta$
for some $\delta\in(0,1)$ (for more details see
also the survey \cite[p. 131]{Tian2012}).

\bexample
\lb{MainExam}

Let $M$ denote the blow-up of $\PP^2$ at three non colinear points.
It is well-known that it admits \KE metrics \cite{Siu,TianYau}.
According to \cite[Theorem 8.4.2]{Dolg}, 
$$
\AutMJz=(\CC^\star)^2.
$$
In particular, $\AutMJz$ is equal to its center.
Thus, Conjecture \ref{TianConj} (ii) fails for $M$
by Theorem \ref{KEGexistenceIntroThm} 
(cf. \cite[Theorem 4.4]{Tian97}). 

\eexample

These results motivate a reformulation of Tian's original conjecture.
Albeit being a purely analytic criterion,
properness should be morally equivalent 
to properness in a metric geometry sense, namely, that the
Mabuchi functional should grow at least linearly relative to some metric
on $\calH$ precisely when a \KE metric exists in $\calH$.
Our goal in this article is to make this intuition rigorous.

To state our results we introduce some of the basic notions.
Recall \eqref{HoEq1}; the space of smooth strictly $\o$-plurisubharmonic 
functions (\K potentials)
\beq
\label{HoEq1}
\calH_\o 
:=
\{\vp\in C^{\infty}(M) \,:\, \,  \omega_\vp\in\calH\}
\end{equation} 
can be identified with $\H\times\RR$.
Consider the following weak 
Finsler metric on $\mathcal H_\o$ \cite{da4}:
\begin{equation}\label{FinslerDef1}
\|\xi\|_\vp:=  V^{-1}\int_M |\xi| \ovpn, 
\q  \xi \in T_\vp \mathcal{H}_\o=C^\infty(M).
\end{equation}
We denote by $d_1:\Ho\times\Ho\ra\RR_+$ the associated
path length pseudometric. According to \cite{da4}
it is a bona fide metric. By looking at level sets of the Aubin--Mabuchi energy, it is possible to embed $\mathcal H$ into $\mathcal H_\o$ (see \eqref{H0Eq}), giving a metric space $(\mathcal H, d_1)$.

Suppose $G$ is a subgroup of $\AutMJz$. We will prove that $G$ acts on $\mathcal H$ by $d_1$-isometries, hence induces a pseudometric on the orbit space $\H/G$,
$$
d_{1,G}(Gu,Gv):=\inf_{f,g\in G}d_1(f.u,g.v).
$$
Following Zhou--Zhu \cite[Definition 2.1]{zztoric} 
and Tian \cite[Definition 2.5]{Tian2012},\cite{Tian2014}, we also define 
the descent of $J$ to $\H/G$,
$$
J_G(Gu):=\inf_{g\in G}J(g.u).
$$

\bdefin
Let $F:\H\ra \RR$ be $G$-invariant.

\noindent $\bullet$ 
We say $F$ is $d_{1,G}$-proper if
for some $C,D >0$,
$$
F(u) \geq C d_{1,G}(G0,Gu) - D.
$$

\noindent $\bullet$  
We say $F$ is $J_{G}$-proper if
for some $C,D >0$,
$$
F(u) \geq C J_{G}(Gu) - D.
$$

\edefin

The following result complements Theorem \ref{KEGexistenceIntroThm} by giving 
yet another optimal replacement for Conjecture \ref{TianConj} (ii) 
in the setting of \KE metrics. The direction (iii) $\;\Rightarrow\;$ (i)
is due to Tian \cite[Theorem 2.6]{Tian2012}. 
In the case of toric Fano manifolds, a variant of the direction 
(i) $\,\Rightarrow\,$ (iii) is due to Zhou--Zhu 
\cite[Theorem 0.2]{zztoric}.

\begin{theorem} 
\label{KEexistenceIntroThm}
Suppose $(M,\JJJ,\o$) is Fano. Set $G:=\Aut(M,\JJJ)_0$.
The following are equivalent:

\smallskip
\noindent (i) 
There exists a K\"ahler--Einstein metric in $\mathcal H$.

\noindent (ii) 
$E$ is $G$-invariant and its descent to the quotient space 
$\calH/G$ is $d_{1,G}$-proper.

\noindent (iii) 
$E$ is $G$-invariant and its descent to the quotient space 
$\calH/G$ is $J_{G}$-proper.
\end{theorem}

\begin{remark}
In both Theorems \ref{KEGexistenceIntroThm} 
and \ref{KEexistenceIntroThm} 
it is possible to replace $E$ with the Ding functional 
\cite{Ding1988}. See 
Theorems \ref{KEGexistenceThm} and \ref{KEexistenceThm}. 
\end{remark}
 
It is interesting to compare Theorem \ref{KEexistenceIntroThm}
with yet another---perhaps more familiar---notion of properness.
Denote by $\Lambda_1$ the real eigenspace of
the smallest positive eigenvalue of $-\Delta_\o$, and set
$$
\calH_\o^\perp:=\{\vp\in\H\,:\, \int \vp\psi\on=0,\;\forall \psi\in\Lambda_1\}.
$$
When $\o$ is \KEno, 
it is well-known that $\Lambda_1$ is in a one-to-one correspondence with
holomorphic gradient vector fields \cite{ga}.
Tian made the following conjecture in the 90's 
\cite[Conjecture 5.5]{Tian97},
\cite[Conjecture 6.23]{Tianbook},\cite[Conjecture 2.15]{Tian2012}.

\bconj
\lb{TianConj2}
Suppose $(M,\JJJ,\o)$ is Fano K\"ahler-Einstein.
Then for some $C,D>0$,
$$
E(\vp) \geq C J(\vp) - D, \qq  \vp \in \H_\o^\perp.
$$
\econj

Conjecture \ref{TianConj2} was originally 
motivated by results in conformal geometry related to 
the determination of the best constants in the 
borderline case of the Sobolev inequality.
By restricting to functions orthogonal to the first eigenspace of the Laplacian,
Aubin was able to improve the constant in the aforementioned inequality
on spheres \cite[p. 235]{Aubinbook}. This can be seen as the sort of coercivity
of the Yamabe energy occuring in the Yamabe problem,
and it clearly fails without the orthogonality assumption due to the presence
of conformal maps. Conjecture \ref{TianConj2} stands in clear analogy with
the picture in conformal geometry, by stipulating that coercivity of
the K-energy holds in `directions perpendicular to holomorphic maps.'
It can be thought of as a higher-dimensional fully
nonlinear generalization of the classical Moser--Trudinger inequality.

It is a rather simple 
consequence of the work of Bando--Mabuchi \cite{BM}
that when a \KE metric exists,
$J_G$-properness implies $J$-properness on $\H_\o^\perp$
\cite[Corollary 5.4]{Tian97},\cite[Lemma A.2]{zztoric},\cite[Theorem 2.6]{Tian2012}. 
Therefore, Theorem \ref{KEexistenceIntroThm} 
resolves Tian's conjecture.

\begin{corollary} 
\lb{TianConj2Cor}
Conjecture \ref{TianConj2} holds.
\end{corollary}

In particular, this provides a new functional inequality on $S^2$.
This inequality seems different from Aubin's well-known inequality 
\cite[Theorem 6.70]{Aubinbook}, especially in view
of Sano's example \cite[Remark 1.1]{Sano}. It would be interesting
to compare it to \cite[Theorem 10.11]{R08}.

Motivated by Theorem \ref{KEexistenceIntroThm} we make the following conjecture.
In the case of \KE metrics, 
Tian \cite[p. 127]{Tian2012} already conjectured the equivalence 
of (i) and (iii). 

\bconj
\lb{MainConj}
Let $(M,\JJJ,\o$) be a compact \K manifold. Set $G:=\Aut(M,\JJJ)_0$.
The following are equivalent:

\smallskip
\noindent (i) 
There exists a constant scalar curvature metric in $\mathcal H$.

\noindent (ii) 
$E$ is $G$-invariant and its descent to the quotient space 
$\calH/G$ is $d_{1,G}$-proper.

\noindent (iii) 
$E$ is $G$-invariant and its descent to the quotient space 
$\calH/G$ is $J_{G}$-proper.
\econj

It would be interesting to
compare Conjecture \ref{MainConj} to
\cite[Conjecture 6.1]{c1}. We refer the reader 
to Remark \ref{Dingd1Remark}.

We reduce Conjecture \ref{MainConj} to a purely PDE problem
of regularity of minimizers. To phrase the result, we denote
by $\E_1$ the $d_1$-metric completion of $\H$.
We denote still by $E$ the greatest $d_1$-lower
semicontinuous extension of $E$ to $\E_1$.
We refer to 
Sections \ref{FinslerSec}--\ref{ActSec} for precise details.

\bconj
\lb{MainHardConj}
Minimizers of $E$ over $\E_1$ are smooth.
\econj

Conjecture \ref{MainHardConj} is inspired by previous work of many authors
and, as recalled in \S\ref{P3SubSec}, 
it is already known in the case of \KE (edge) metrics or \KR solitons
by combining previously known results.
Observe that a resolution of Conjecture \ref{MainHardConj} 
would also imply
\cite[Conjecture 6.3]{c1}.

The modified properness conjecture follows from the previous one.

\bthm
\lb{MainReducThm}
Conjecture \ref{MainHardConj}
implies
Conjecture \ref{MainConj}.
\ethm

We also establish sharp versions of the properness conjecture in the setting of \KE edge metrics and \KR solitons. Denote by $\Aut^X(M,\JJJ)_0$ the subgroup of $\Aut(M,\JJJ)_0$ defined in 
\eqref{AutMJXEq}. For a precise statement of the following
result see Theorem \ref{solitonKEexistenceThm}.

\begin{theorem} 
\label{KRSexistenceIntroThm}
Suppose $(M,\JJJ,\o$) is Fano and let $X\in\aut(M,\JJJ)$. 
Set $G:=\Aut^X(M,\JJJ)_0$.
Then a version of Theorem \ref{KEexistenceIntroThm} holds
both for the Tian--Zhu modified K-energy $E^X$ and 
for the modified Ding functional $F^X$.
\ethm

A version of Theorem \ref{KRSexistenceIntroThm} in the absence
of holomorphic vector fields is due to \cite{ctz}.

Denote by $\Aut(M,D,\JJJ)_0$ the subgroup of $\Aut(M,\JJJ)_0$ defined in \eqref{AutMDEq}.
For a precise statement of the following
result see Theorem \ref{KEEexistenceThm}.

\begin{theorem} 
\label{KEEexistenceIntroThm}
Suppose $(M,\JJJ,\o$) is a compact K\"ahler manifold and a smooth divisor $D\subset M$
satisfying $c_1(M)-(1-\be)[D]=[\o]$ for some $\be\in(0,1)$.
Set $G:=\Aut(M,D,\JJJ)_0$.
Then a version of Theorem \ref{KEexistenceIntroThm} holds
both for the twisted K-energy $E^\be$ and 
for the twisted Ding functional $F^\beta$.
\ethm

In the absence of holomorphic vector fields, i.e., when $G$ is trivial, 
some versions
of Theorem \ref{KEEexistenceIntroThm} exist in the literature.
The direction (iii) $\;\Rightarrow\;$ (i) is due to 
\cite[Theorem 2]{JMR}.
A version of the direction (i) $\;\Rightarrow\;$ (iii)
in the special case of $D$ plurianticanonical on a Fano manifold
(which implies the triviality of $G$) is due to \cite[Proposition 3.6]{cds},
\cite[Corollary 2.2]{T15},\cite[Theorem 0.1]{TZ15}.

\bremark
Versions of Corollary \ref{TianConj2Cor} for \KE edge metrics and \KR
solitons also follow from our work. We omit the statements for brevity.
\eremark

\subsection{Sketch of the proofs}

As already noted, much work has gone into showing different versions of Tian's conjecture. 
These works are mostly based on the continuity method,
the Ricci flow and the $J$-flow. 
To cite a few papers from a rapidly growing literature, 
we mention \cite{ctz,pssw,SW,CJflow, SWJflow, RJFA,zztoric,zz} and references therein.

Perhaps one of the main thrusts of the present article is that it is considerably more powerful to use the Finsler geometry of $\H$ and various of its subspaces to treat in a unified manner essentially all instances of the properness conjectures. In this spirit, we state a completely general 
existence/properness principle
(Theorem \ref{ExistencePrinc}).
This principle is stated in terms of four pieces of data: the space of regular candidates
$\calR$, a metric $d$ on it, a functional $F$ defined on $\calR$, and a group
$G$ acting on  $\calR$.
The data is assumed to satisfy four axioms, or conditions, that we denote by \ref{a1}--\ref{a4};
we refer the reader to Notation \ref{MainNot}.
Roughly stated, the existence principle guarantees $F$ is $d_G$-proper on $\calR$
if and only if the data satisfies seven additional properties that we denote by
\ref{p1}--\ref{p7}; we refer the reader to Hypothesis \ref{MainHyp}.

Property \ref{p1} guarantees $F$ is convex along `sufficiently many'
$d$-geodesics, and is inspired by the work of Berndtsson \cite{brn}.
Property \ref{p2} is a compactness requirement for minimizing sequences,
and is inspired by the work of Berman et al. \cite{bbegz,bbgz}.
Property \ref{p3} is a regularity requirement for minimizes, and is inspired by the work of Berman and Berman--Witt-Nystr\"om \cite{brm1,bwn}.
Property \ref{p4} stipulates that $G$ act by $d$-isometries.
Property \ref{p5} requires that $G$  act transitively on the set of minimizes,
and is inspired by the classical work of Bando--Mabuchi and its recent generalizations by Berndtsson and others. 
Property \ref{p6} seems to be a new ingredient. It requires
that in the presence of minimizers the pseudometric $d_G$  is realized by elements in each orbit.
Property \ref{p7} is standard and requires the `cocycle' functional associated to $F$ to be $G$-invariant.

Section \ref{GenSec} is devoted to the proof of the existence principle. Section
\ref{FinslerSec}
recalls necessary preliminaries concerning the Finsler geometry from \cite{da4}. Section \ref{ActSec} is rather lengthy and studies basic properties of energy functionals relative to the Finsler metric completion. In particular we verify condition \ref{a2} and properties \ref{p3} and
\ref{p4} in the particular cases of interest. Section \ref{CartanSec} is devoted to the proof of property \ref{p6}. 
A technical input here is a partial Cartan decomposition
(Proposition \ref{PartialCartanProp}). Finally, in the remaining sections we prove various instances of the properness conjecture. Section 
\ref{KESec} contains the proof of results
containing Theorems 
\ref{KEGexistenceIntroThm} and \ref{KEexistenceIntroThm}. 
Section \ref{KRSSec} contains the proof of Theorem
\ref{KRSexistenceIntroThm}.
 Section \ref{KEESec} contains the proof of Theorem  \ref{KEEexistenceIntroThm}. 
 Section \ref{cscSec} contains the proof of Theorem \ref{MainReducThm}.

\section{A general existence/properness principle}
\lb{GenSec}

\bnot
\label{MainNot}
The data $(\mathcal R,d,F,G)$ is defined as follows.

\begin{enumerate}[label = (A\arabic*)]
  \item\label{a1} $(\mathcal R,d)$ is a metric space with a special element $0\in\mathcal R$,
        whose metric completion is denoted $(\overline{\mathcal R},d)$.
  \item\label{a2} $F : \mathcal R \to \Bbb R$ 
is lower semicontinuous (lsc). Let  
        $F: \overline{\mathcal R} \to \Bbb R \cup \{ +\infty\}$ 
        be the largest lsc extension
        of $F: \mathcal R \to \Bbb R$:
$$F(u) = \sup_{\varepsilon > 0 } \bigg(\inf_{\substack{v \in \mathcal R\\ d(u,v) \leq \varepsilon}} F(v) \bigg), \ \ u \in \overline{ \mathcal R}.$$

For each 
$u,v\in{\mathcal R}$ define also
$$
F(u,v):=F(v)-F(u).
$$
  \item\label{a3} 
The set of minimizers of $F$ on $\overline{\mathcal R}$ is denoted
$$
\mathcal M:= 
\Big\{ u \in \overline{\mathcal R} \ : \ F (u)= 
\inf_{v \in \overline{\mathcal R}} F(v)
\Big\}.
$$
  \item\label{a4} Let $G$ be a group  
  acting on ${\mathcal R}$ by 
  $G\times{\mathcal R}\ni(g,u) \to g.u\in {\mathcal R}$. Denote by
${\mathcal R} /G$ the orbit space, by $Gu\in{\mathcal R} /G$ the 
  orbit of $u\in{\mathcal R}$, 
  and~define~$d_G:{\mathcal R}/G\times {\mathcal R}/G\ra\RR_+$ by
$$
d_G(Gu,Gv):=\inf_{f,g\in G}d(f.u,g.v).
$$
\end{enumerate}
\enot

\bhyp
\label{MainHyp}
The data $(\mathcal R,d,F,G)$
satisfies the following properties.

\begin{enumerate}[label = (P\arabic*)]
  \item\label{p1}
 For any $u_0,u_1 \in \mathcal R$ there exists a $d$--geodesic segment $[0,1] \ni t \mapsto u_t \in \overline{\mathcal R}$ for which
 $t\mapsto F(u_t) \textup{ is continuous and convex on }[0,1].$
  \item \label{p2} 
If  $\{u_j\}_j\subset \overline{\mathcal R}$  satisfies 
$\lim_{j\ra\infty}F(u_j)= \inf_{\overline{\mathcal R}} F$, and
for some $C>0$, $d(0,u_j) \leq C$ for all $j$, then there exists a $u\in \mathcal M$ and a subsequence $\{u_{j_k}\}_k$ 
$d$-converging to $u$.
  \item\label{p3} 
  $\mathcal M \subset \mathcal R.$
  \item\label{p4} $G$ acts on ${\mathcal R}$ by $d$-isometries.
  \item \label{p5} $G$ acts on $\mathcal M$ transitively.
  \item \label{p6} If $\mathcal M \neq \emptyset$,  
then for any $u,v \in \mathcal R$ there exists $g \in G$
such that $d_G(Gu,Gv)=d(u,g.v)$.

  \item \label{p7} For all $u,v \in \mathcal R$ and $g \in G$,
$F(u,v)=F(g.u,g.v)$.

\end{enumerate}

\ehyp

We make two remarks. 
First, by \ref{a2}, 
\begin{equation}
\label{InfRealization}
\inf_{v \in \overline{\mathcal R}} F(v)
=
\inf_{v \in \mathcal R} F(v).
\end{equation}
Second, thanks to \ref{p4} and the next lemma,
the action of $G$, originally defined on $\calR$ in \ref{a4}, 
extends to an action of $G$ by $d$-isometries on the metric completion $\overline{\calR}$.

\begin{lemma}
\label{LipschitzExt} 
Let $(X,\rho)$ and $(Y,\delta)$ be two complete metric spaces, $W$ a dense subset of $X$ and $f:W \to Y$ a $C$-Lipschitz function, i.e.,
\begin{equation}\label{lipineq}
\delta(f(a),f(b)) \leq C \rho(a,b), \q \forall\, a,b \in W.
\end{equation}
Then $f$ has a unique $C$-Lipschitz continuous extension 
to a map $\bar f:X\ra Y$.
\end{lemma}
\begin{proof}Let $w_k \in W$ be a Cauchy sequence converging to some $w \in X$. Lipschitz continuity gives
$$\delta(f(w_k),f(w_l)) \leq C \rho(w_k,w_l),$$
hence $\bar f(w) := \lim_k f(w_k) \in Y$ is well defined and independent of the choice of approximating sequence $w_k$. Choose now another Cauchy sequence $z_k \in W$ with limit $z \in X$, plugging in $w_k,z_k$ in \eqref{lipineq} and taking the limit gives that $\bar f: X \to Y$ is $C$-Lipschitz continuous.  
\end{proof}

The following result will provide the framework that relates existence of canonical \K metrics
to properness of functionals with respect to the Finsler metric.

\begin{theorem}
\label{ExistencePrinc}
Let $(\mathcal R,d,F,G)$ be as in Notation \ref{MainNot}
and satisfying Hypothesis \ref{MainHyp}.
The following are equivalent:
\begin{itemize}
\item[(i)] $\mathcal M$ is nonempty.
\item[(ii)]$F:{\mathcal R}\ra \RR$ 
is $G$-invariant, and
for some $C,D>0$,
\begin{equation}\label{Dproperness}
    F(u) \geq Cd_G(G0,Gu)-D,\q \h{for all\ } u \in \mathcal R.
\end{equation}
\end{itemize}
\end{theorem}

\bremark
The $G$-invariance condition can be considered as 
a version of the Futaki obstruction \cite{Fut}.
\eremark

\bpf

\noindent
{\it (ii)\ $\Rightarrow$\ (i)}. If condition (ii) holds, then  $F$ is bounded from below. 
By \eqref{InfRealization}, \eqref{Dproperness}, 
the $G$--invariance of $F$ and the definition of $d_G$ there 
exists $u_j \in \mathcal R$ such that 
$\lim_{j}F(u_j) = \inf_{\overline{\mathcal R}} F$ 
and $d(0,u_j) \leq d_G(G0,Gu_j) + 1<C$ for $C$ independent of $j$. 
By \ref{p2}, $\mathcal M$ is non-empty.

\medskip

\noindent
{\it (i)\ $\Rightarrow$\ (ii)}.
We start with a standard lemma.

\blem
\noindent
(i) If  \ref{p4} holds, 
$(\overline{\mathcal R}/G,d_G)$ and $(\mathcal R/G,d_G)$ are pseudo-metric spaces.

\noindent
(ii) If $\mathcal M \neq \emptyset$, \ref{p4} and \ref{p6} hold, 
$(\mathcal R/G,d_G)$ is a metric space.

\elem

\bpf
(i) It is enugh to show that $(\overline{\mathcal R}/G,d_G)$ is a pseudo-metric space. Using \ref{p4} and the fact that $d$ is symmetric,
$$
d_G(Gu,Gv):=\inf_{f,g\in G}d(f.u,g.v)=\inf_{h\in G}d(u,h.v)
=\inf_{h\in G}d(h.v,u)=d_G(Gv,Gu).
$$
Thus, $d_G$ is symmetric.

Since $d$ is nonnegative, 
given $u,v,w \in \overline{\mathcal R}$ and $\epsilon >0$, there exist $f,g \in G$ such that 
$d_G(Gu,Gw) > d(f.u,w)-\epsilon$ 
and $d_G(Gv,Gw)> d(g.v,w)-\epsilon$. 
The triangle inequality for $d$ and \ref{p4} give
\begin{flalign*}
d_{G}(Gu,Gv)&\leq 
d(u,f^{-1}g.v)=d(f.u,g.v) \cr
&\leq d(f.u,w)+d(g.v,w) < 2\varepsilon+d_{G}(Gu,Gv) + d_{G}(Gv,Gw).
\end{flalign*}
Letting $\eps$ tend to zero shows $d_G$ satisfies the triangle inequality.
Thus $d_G$ is a pseudo-metric.

\medskip

\noindent
(ii)
Since $d$ is nonnegative so is $d_G$. Now, let $u,v \in \mathcal R$ satisfy $d_G(Gu,Gv)=0$. 
By \ref{p6}, 
$d(u, f.v)=0$
for some $f \in G$. Since $d$ is a metric, $u = f.v$, hence $Gu=Gv$.
\epf

A geodesic in $(\overline{\mathcal R},d)$ need not
descend to a geodesic in $(\overline{\mathcal R}/G,d_G)$. Even when a geodesic does
 descend, its speed may not be the same in
the quotient space. A simple example in our cases of interest is  a one-parameter subgroup of $G$ acting on a fixed element 
$u\in \mathcal R$, whose descent is a trivial geodesic.

The next lemma gives a criterion for when a geodesic in 
$(\overline{\mathcal R},d)$ descends to a geodesic in $(\overline{\mathcal R}/G,d_G)$
with the same speed.

\blem
\label{GeodDescentLemma}
 Suppose that  $\mathcal M \neq \emptyset$, \ref{p4} and \ref{p6} hold. 
Let $u_0,u_1 \in \mathcal R$ satisfy 
$d_G(Gu_0,Gu_1)=d(u_0,u_1)$, and let 
$\{u_t\}_{t\in[0,1]} \subset \overline{\mathcal R}$ be a $d$-geodesic connecting 
$u_0$ and $u_1$. 
Then, $\{Gu_t\}_{t\in[0,1]} \subset \overline{\mathcal R}/G$ is a $d_G$-geodesic 
satisfying
$$
d_G(Gu_a,Gu_b)=d(u_a,u_b)=|b-a|d(u_0,u_1), \q \forall\, a,b \in [0,1].
$$
\elem

\begin{proof} 
Since $d(u_0,u_a)+ d(u_a,u_b)+d(u_b,u_1)=d(u_0,u_1)$ we can write
$$
\baeq
d_{G}([u_0],[u_1]) 
&\leq d_{G}([u_0],[u_a])+ d_G([u_a],[u_b])+d_{G}([u_b],[u_1]) 
\cr
&\leq d(u_0,u_a)+ d(u_a,u_b)+d(u_b,u_1)
\cr
&=d(u_0,u_1)=d_G([u_0],[u_1]).
\eaeq
$$
Hence, there is equality everywhere, so $d_{G}([u_a],[u_b])=d(u_a,u_b)=|a-b|d(u_0,u_1)$.
\end{proof}

\blem
\label{GinvLemma}
Suppose that (i) holds. Then $F:\mathcal R \to \Bbb R$ is $G$--invariant. 
\elem

\bpf
By assumption, $\mathcal M$ is nonempty. Let $v\in\mathcal M$.
By \ref{p3}, $v\in\mathcal R$.
By \ref{p4}, 
$f.v \in \mathcal M$ for any $f \in G$. 
Thus, $F(v)=F(f.v)$. 
By \ref{a3}, subtracting $F(u)$ from both sides,
$F(u,v)=F(u,f.v)$
for any $u\in \overline{\mathcal R}$. By \ref{p7},
$F(u,v)=F(f^{-1}.u,v)$, so 
adding $F(v)$ to both sides yields that
$F(u)=F(f.u)$ for every $f\in G$.
\epf

By the above lemma it makes sense to introduce $F_G:\mathcal R/G \to \Bbb R$, the descent of $F$ to the quotient $\mathcal R/G$. Let 
$v \in \mathcal M \subset \mathcal R$. Define, 
$$
C:= 
\inf 
\bigg\{
\frac{F_G(Gv,Gu)}{d_G(Gv,Gu)}
\,:\, u \in \mathcal R,\; d_G(Gv,Gu) \geq 1\bigg\}.
$$
If $C>0$, then we are done. Suppose  $C=0$. Then by \ref{p4} there exists $u(k) \in \mathcal R$ such that $F_G(Gv,Gu(k))/{d_G(Gv,Gu(k))} \to 0$ and 
$d(v,u(k))=d_G(Gv,Gu(k)) \geq 1$. 
By Lemma \ref{GinvLemma}, in fact
$F_G(Gv,Gu)=F(v,u)$. Thus, 
$$
\frac{F(v,u(k))}{d(v,u(k))} \to 0.
$$
Using \ref{p1}, 
let $[0,d(v,u(k))] \ni t \mapsto u(k)_t \in \overline{\mathcal R}$ be a unit 
speed $d$-geodesic connecting $u(k)_0=v$ and $u(k)_{d(v,u(k))}=u(k)$ 
such that $t \mapsto F(u(k)_t)$ is convex. 
As $v$ is a minimizer of $F$, by convexity we obtain 
\beq
\label{ConvConseqEq}
0 \leq 
F(u(k)_1) -F(v)\leq  
\frac{F(v,u(k))}{d(v,u(k))} 
\to 0.
\eeq
As $d(v,u(k)_1)=1$, \ref{p2} and \eqref{ConvConseqEq} imply that after perhaps passing to a subsequence of $u(k)_1$ we have $d(u(k)_1,\tilde v) \to 0$ for some $\tilde v \in \mathcal M$ . By 
\ref{p5}, $\tilde v=f.v$ for some $f \in G$. 
Now, 
\beq
\label{TriIneqInKEq}
0=d(f.v,\tilde v) \geq d(f.v,u(k)_1)-d(u(k)_1,\tilde v).
\eeq
By Lemma \ref{GeodDescentLemma},
$d_G(Gv,Gu(k)_1)=d(v,u(k)_1)=1$.
Thus, $d(f.v,u(k)_1) \geq d_G(Gv,Gu(k)_1)=1$.
Since $d(u(k)_1,\tilde v) \to 0$, 
it follows that 
$d(f.v,\tilde v) \geq 1$,
a contradiction with \eqref{TriIneqInKEq}.
Thus (ii) holds, 
concluding the proof of Theorem \ref{ExistencePrinc}.
\epf

\begin{remark} 
\lb{MainThmRemark}
(i) The first direction  
in the above proof only uses the compactness condition~\ref{p2}.

\noindent
(ii)
By density, equation \eqref{Dproperness}  is in fact equivalent to
\begin{equation}\label{DGinvproperness}
    F(u) \geq Cd_G(G0,Gu)-D,\q \h{for all\ } u \in \overline{\mathcal R}.
\eeq

\noindent
(iii) In this article we will always verify a stronger condition
than \ref{p6}, namely that for every $u,v\in \overline{\mathcal R}$
there exists $g\in G$ such that $d(u,g.v)=d_G(Gu,Gv)$.

\end{remark}

Theorem \ref{ExistencePrinc} and the compactness condition \ref{p2} 
have the following consequence.

\begin{corollary}
\label{superimportantcorollary:)}
Suppose $\mathcal M$ is nonempty. Then if $u_j \in \mathcal R$ satisfies $F(u_j) \to \inf_{{\mathcal R}}F$, then there exists $g_{j} \in G$  and $ u \in \mathcal M$ such that $g._{j} u_{j} \to_{d} u$.
\end{corollary}

In the search for canonical K\"ahler metrics one often studies geometric data that minimizes in the limit an appropriate energy functional $F$. 
Examples include the Ricci flow along with its twisted and modified versions, the Ricci iteration, or the (weak) Calabi flow. As stated, the above result partially generalizes  \cite[Theorem B (ii)]{bbegz} and \cite[Theorem 1.2]{brm1}, each of which treats a particular case of interest in the presence of a trivial automorphism group. 

To mention a concrete application, as it will be clear after the proof of Theorem \ref{KEexistenceThm}, the above corollary gives $d_1$-convergence up to automorphisms of the Ricci iteration in the presence of a \KE metric. Smooth convergence up to automorphisms was conjectured in \cite[Conjecture 3.2]{R08}.

\section{The Finsler geometry}
\lb{FinslerSec}
Let $(M,\JJJ,\o)$ denote a connected compact closed \K manifold. 
Recall \eqref{HoEq}; the space of smooth strictly $\o$-plurisubharmonic 
functions (\K potentials)
\beq
\label{HoEq}
\calH_\o 
:=
\{\vp\in C^{\infty}(M) \,:\, \,  \omega_\vp\in\calH\},
\end{equation} 
can be identified with $\H\times\RR$.
For any $\varphi \in \mathcal H_\o$ the total volume
\beq\lb{VEq}
V:=\int_M\o_\varphi^n.
\eeq
is constant independent of $\varphi$. Consider the following weak 
Finsler metric on $\mathcal H_\o$ \cite{da4}:
\begin{equation}\label{FinslerDef}
\|\xi\|_\vp:=  V^{-1}\int_M |\xi| \ovpn, 
\q  \xi \in T_\vp \mathcal{H}_\o=C^\infty(M).
\end{equation}

\bremark
More generally, one may consider $L^p$ metrics on $\mathcal H_\o$
\cite{da4}. The case $p=2$ is the much-studied weak Riemannian metric 
of Mabuchi mentioned in the Introduction. 
Though it may seem surprising at first, in this article
we only need an understanding of the case $p=1$. 
\eremark

A curve $[0,1]\ni t \to \alpha_t \in \mathcal{H}$ 
is called smooth if $\alpha(t,z)=\alpha_t(z) \in C^\infty([0,1] \times M)$. 
Denote $\dot\a_t:=\del\a(t)/\del t$.
The length of a smooth curve $t \to \alpha_t$ is
\begin{equation}\label{curve_length_def}
\ell_1(\alpha):=\int_0^1\|\dot \alpha_t\|_{\alpha_t}dt.
\end{equation}

\bdefin
The path length distance of $(\Ho,||\,\cdot\,||)$ is defined by
$$
d_1({u_0},{u_1}):= 
\inf\{\ell_1(\alpha)\,:\,\alpha:[0,1]\ra\Ho \h{\ is a smooth curve with \ }
\a(0)=u_0,\, \a(1)=u_1\}.
$$
We call the pseudometric $d_1$ the {\it Finsler metric}. 
\edefin

It turns out $d_1$ is a bona fide metric
\cite[Theorem 3.5]{da4}.
To state the result,
consider 
$[0,1]\times\RR\times M$ as a complex manifold of dimension
$n+1$, and denote by $\pi_2:[0,1]\times\RR\times M\ra M$ the natural projection.
\bthm
\lb{d1Thm}
$(\Ho, d_1)$ is a metric space.
Moreover, 
\begin{equation}
\label{distgeod}
d_1(u_0,u_1)=\|\dot u_0 \|_{u_0}\ge0,
\end{equation}
with equality iff $u_0=u_1$, 
where $\dot u_0$ is the image of $(u_0,u_1)\in \Ho\times\Ho$ under the 
Dirichlet-to-Neumann map for the \MA equation,
\beq\label{MabuchiEq}
\vp\in\PSH(\pi_2^\star\o, [0,1]\times\RR\times M),\q
(\pi_2^\star\o+\i\ddbar \vp)^{n+1}=0, 
\q
 \vp|_{\{i\}\times\RR}=u_i,\; i=0,1.
\eeq
\ethm
The Dirichlet-to-Neumann operator simply maps $(u_0,u_1)$
to the initial tangent vector of the curve $t\mapsto u_t$
that solves \eqref{MabuchiEq}.
Since $u$ is $\pi_2^*\o$-psh and independent of the imaginary part
of the first variable, it is convex in $t$. Thus,
\beq
\lb{dotu0eq}
\dot u_0(x):=\lim_{t\ra 0^+}\frac{u(t,x)-u_0(x)}{t},
\eeq
with the limit well-defined since the difference quotient
is decreasing in $t$. Let
$$
\PSH(M,\o)=\{ \vp \in L^1(M,\o^n)\,:\, \h{$\vp$ is upper semicontinuous and } 
\ovp\geq 0 \}.
$$
Following Guedj--Zeriahi \cite[Definition 1.1]{gz} define,
$$
\E(M,\o):=
\big\{
\vp\in \PSH(M,\o)\,:\,
\lim_{j\ra-\infty}\int_{\{\vp\le j\}}(\o+\i\ddbar\max\{\vp,j\})^n= 0
\big\}.
$$
For each $\vp\in\E(M,\o)$, define
$$
\ovpn:=\lim_{j\ra-\infty}
{\bf 1}_{\{\vp>j\}}
(\o+\i\ddbar\max\{\vp,j\})^n.
$$
By definition, ${\bf 1}_{\{\vp>j\}}(x)$ is equal to $1$ if $\vp(x)>j$ and zero otherwise,
and
the measure $(\o+\i\ddbar\max\{\vp,j\})^n$ is defined by the
work of Bedford--Taylor \cite{bt} since $\max\{\vp,j\}$ is bounded.
Define,
$$
\E_1:=\big\{\vp\in\E(M,\omega)\,:\, \int|\vp|\ovpn<\infty\big\}.
$$
The next result characterizes the $d_1$-metric completion
\cite[Theorem 2]{da4}.
\bthm
\lb{d1CompletionThm}
The metric completion of $(\H_\o,d_1)$ equals $(\E_1,{d_1})$,
where 
$$
d_1(u_0,u_1):=\lim_{k\ra\infty}
d_1(u_0(k),u_1(k)),
$$
for any smooth decreasing sequences $\{u_i(k)\}_{k\in\NN}\subset\Ho$
converging pointwise to $u_i \in \mathcal E_1, i=0,1$.
Moreover, for each $t\in(0,1)$, define
\begin{equation}\label{EpGeodDef}
u_t:= \lim_{k \to\infty}u_t(k), \ t \in (0,1),
\end{equation}
where $u_t(k)$ is the solution of \eqref{MabuchiEq}
with endpoints $u_i(k), i=0,1$. Then $u_t\in \E_1$,
and the curve $t \to u_t$ is well-defined independently of the choices
of approximating sequences and is a $d_1$-geodesic.
\ethm

\section{Action functionals and their Euler--Lagrange equations
}
\lb{ActSec}

In \S\ref{EnergySubSec}--\ref{ActionSubSec} we review certain 
energy functionals on $\H$ and $\Ho$. For an expository survey 
of this topic we refer to \cite[\S5]{R14}. 

\subsection{Basic energy functionals}
\lb{EnergySubSec}

The most basic functional, introduced by Aubin
\cite{Aubin84}, is defined by
\beq
\label{AubinEnergyEq}
\begin{aligned}
J(\vp)=J(\o_\vp):=
V^{-1}\int_M\vp\o^n
-
\frac{V^{-1}}{n+1}\int_M
\vp\sum_{l=0}^{n}\o^{n-l}\w\o_{\vp}^{l}.
\end{aligned}
\eeq
The notation $J(\vp)=J(\o_{\vp})$ is justified by the fact that $J(\vp)=J(\vp+c)$
for any $c\in\RR$.
The Aubin--Mabuchi functional was introduced by Mabuchi \cite[Theorem 2.3]{Mabuchi87},
\begin{equation}\label{AMdef}
\h{\rm AM}(\vp):= V^{-1}\int_M\vp\o^n-J(\vp)=
\frac{V^{-1}}{n+1}\sum_{j=0}^{n}\int_M \vp\, \o^j \wedge \o_\vp^{n-j},
\end{equation}
Note that
\beq
\lb{AMDifferenceEq}
\h{\rm AM}(v)-\h{\rm AM}(u)
=
\frac{V^{-1}}{n+1}\int_M(v-u)\sum_{k=0}^n \o_u^{n-k}\w \o_v^k.
\eeq
Among other things, this formula shows that 
\beq\lb{AMMonEq}
u \leq v \q\Rightarrow \q \h{\rm AM}(u) \leq \h{\rm AM}(v).
\eeq

One can characterize $d_1$-convergence very concretely, as elaborated below. In addition, we note that monotone sequences with limit in $\mathcal E_1$ are always $d_1$-convergent \cite[Proposition 5.9, Proposition 4.9]{da4}:

\blem
\label{L1Lemma}
Suppose that $\{u_k\}\subset\E_1$ and $u\in \E_1$. Then the following hold: \\
\noindent 
(i) 
$d_1(u_k,u) \to 0$ if and only if $\h{\rm AM}(u_k) \to \h{\rm AM}(u)$
and $u_k \to u$ in $L^1(M,\on)$. \\
\noindent 
(ii) 
If $\{u_k\}$ increases (or decreases) pointwise a.e. to $u$ then $d_1(u_k,u) \to 0$.
\elem

The subspace
\beq
\lb{H0Eq}
\calH_0:=\h{\rm AM}^{-1}(0)\cap \Ho
\eeq
is isomorphic to $\calH$ \eqref{HEq}, the space of K\"ahler metrics. We use this isomorphism to endow
$\calH$ with a metric structure, by pulling 
back the Finsler metric defined on $\Ho$. 

\blem
\lb{E1capH0Lemma}
(i)
$\h{\rm AM},J: \mathcal H_\o \to \Bbb R$ each admit a unique $d_1$-continuous extension 
to $\E_1$ using the same formula as \eqref{AMdef} and \eqref{AubinEnergyEq}.

\noindent 
(ii)
$\h{\rm AM}$ is linear along the $d_1$-geodesic $t \to u_t$ defined in \eqref{EpGeodDef}.

\noindent 
(iii)
The subspace $(\E_1\cap \h{\rm AM}^{-1}(0),d_1)$ is a complete geodesic metric space, coinciding with the metric completion of $(\mathcal H_0,d_1)$ (recall 
\eqref{H0Eq}).
\elem

\bpf
(i) In \cite{begz} it is shown that for $u_{1},\ldots,u_k\in\E_1$,
the positive currents
$$
u_1\o_{u_2}^{j_2} \w \o_{u_3}^{j_3} \w \ldots \wedge \o_{u_k}^{j_k}
$$
can be defined by approximating $u_i$ by a decreasing sequence of 
smooth functions $\{u_i(k)\}\subset\mathcal H_\o$ so that the limiting measure
is independent of the choice of such sequences. Thus, 
formula \eqref{AMdef} makes sense for all $u\in\E_1$.
On the other hand, decreasing sequences converge in $d_1$ by 
Lemma \ref{L1Lemma} (ii).
Thus, to prove (i) for AM it remains 
(by Lemma \ref{LipschitzExt} and Theorem \ref{d1CompletionThm}) 
to show that $\h{\rm AM}$ is $d_1$-Lipschitz continuous as a function
from $\H_\o$ to $\RR$. This is proved below in Lemma \ref{AMXLemma} (take $X=0$).
 
Next, Lemma \ref{L1Lemma} (i) gives that $u \to \int_X u \o^n$ is also 
$d_1$-continuous. Thus, \eqref{AMdef} implies that also $J$ admits
a unique $d_1$-continuous extension to $\E_1$, and \eqref{AubinEnergyEq} still holds for the extension.

\medskip
\noindent(ii)
 It is well known that $t \to \h{\rm AM}(u_t)$ is linear for solutions of \eqref{MabuchiEq} when $u_0,u_1 \in \mathcal \E_1\cap L^\infty(M)$
\cite[Remark 4.5]{BBouc}.

When $u_0,u_1 \in \mathcal E_1$,  let $u_t(k)$ be the sequence of curves in the definition of $u_t$ \eqref{EpGeodDef}. For each $t$,
the sequence $u(k)_t$ decreases pointwise to $u_t$, hence by Lemma \ref{L1Lemma}(ii) we have $d_1(u(k)_t,u_t) \to 0$. As $\h{\rm AM}$ is $d_1$-continuous, this gives $\h{\rm AM}(u(k)_t) \to \h{\rm AM}(u_t)$. By the above we also have that  $t \to \h{\rm AM}(u(k)_t)$ is linear. Taking the limit $k \to \infty$ we can conclude that $t \to \h{\rm AM}(u_t)$ is also linear.

\medskip
\noindent(iii) As $\h{\rm AM} : \mathcal E_1 \to \Bbb R$ is $d_1$-continuous, it follows that $\mathcal E_1 \cap \h{\rm AM}^{-1}(0)$ is $d_1$-closed. From (ii) it follows that $(\mathcal E_1 \cap \h{\rm AM}^{-1}(0),d_1)$ is a geodesic metric space.
\epf

As proposed in \cite{begz}, using monotonicity \eqref{AMMonEq} it is possible to extend $\h{\rm AM}$  further to a functional on $\textup{PSH}(M,\o)$ taking $-\infty$ as a possible value \cite{begz}:
\beq
\lb{AMExteq}
\h{\rm AM}(\vp):= \lim_{k \to -\infty}\h{\rm AM}(\max(\vp,k)), \ \vp \in \textup{PSH}(M,\o).
\eeq
By Lemma \ref{L1Lemma},  $d_1(\max(\vp,k),\vp) \to 0$ when $\vp \in  \mathcal E_1$, hence this extension agrees with the one given the previous lemma. In fact the following
result from \cite{begz}:

\begin{lemma} \label{AMfinite} Suppose $\vp \in \textup{PSH}(M,\o)$. Then  $\vp \in \mathcal E_1$ if and only if  $\h{\rm AM}(\vp) > -\infty$.
\end{lemma}

\begin{proof} 
We can suppose that $\sup_M \vp =0$ and denote $\vp_k = \max(\vp,-k)$. 
By Bedford--Taylor \cite{bt}, we can define a functional 
$I:\E_1\cap L^\infty\ra\RR$ by
$$
I(u):=\V\int_M u(\on-\o_u^n).
$$
Recall that for $u\in \E_1\cap L^\infty$ \cite{Tianbook,bt}
\beq
\lb{IJIEq}
\h{\rm AM}(u)
=(I-J)(u)+\V\int u\o_u^n,
\eeq
and since $\int u\o_u^j\w \o^{n-j}\le \int u\o_u^{j-1}\w \o^{n-j+1}$
(integration by parts is again justified by \cite{bt}) \cite{Aubin84}
$$
0\le (I-J)(u)\le \frac n{n+1}I(u).
$$
Thus,
$$ 
\V\int_M \vp_k \o_{\vp_k}^n \leq \textup{AM}(\vp_k) 
\leq 
\frac{\V}{n+1}\int_M \vp_k \o_{\vp_k}^n + \frac{n\V}{n+1}\int_M \vp_k \o^n.
$$
For $k$ big enough $\sup_M \vp_k=0$, hence by \eqref{GreenEq} below, the rightmost term in the above estimate is uniformly bounded.
Lastly, \cite[Proposition 1.4]{gz} gives that $-\int_M |\vp| \o^n_\vp = \lim_k \int_M \vp_k \o_{\vp_k}^n$, concluding the proof.
\end{proof}

Observe that,
\beq
\lb{JH0Eq}
J(\vp)=V^{-1}\int_M\vp\o^n, \q \vp\in\calH_0.
\eeq

Recall that Green's formula implies that 
for all $u\in \textup{PSH}(M,\o)$, 
there exists a constant $C>0$ depending only on $(M,\o)$ such that
\cite[p. 49]{Tianbook}
\beq
\lb{GreenEq}
\sup_M u \le V^{-1}\int_M u\on + C \leq \sup_M u +C.
\eeq

Next, we recall a concrete formula for the $d_1$ metric relating it to the Aubin--Mabuchi energy and also give a concrete growth estimate for $d_1$. First we need to introduce the following rooftop type 
envelope for $u,v \in \mathcal E_1$:
$$
P(u,v)(z):= 
\sup\big\{w(z)\,:\, w \in \PSH(M,\o),\, w \leq \min\{u,v\}\big\}.
$$ 
Note that $P(u,v) \in \mathcal E_1$ \cite[Theorem 2]{da3}.
We recall the following properties of $d_1$ \cite[Corollary 4.14, Theorem 3]{da4}.
\bprop
\lb{PythProp}
Let $u,v\in\E_1$. Then,
\begin{equation}
\label{Pythagorean}
d_1(u,v)=\h{\rm AM}(u) + \h{\rm AM}(v) - 2\h{\rm AM}(P(u,v)).
\end{equation}
Also, there exists $C>1$ such that for all $u,v\in\E_1$,
\begin{equation}
\label{d1CharFormula}
C^{-1} d_1(u,v) \leq \int_M |u-v|\o_u^n + \int_M |u-v|\o_v^n
\le C d_1(u,v).
\end{equation}
\eprop

The following result is stated in \cite[Remark 6.3]{da4}.
As it will be essential for us, we give a proof here.

\begin{proposition} 
\label{Jproperness}
There exists $C',C>1$ such that
for all $u \in \mathcal H_0$ (recall \eqref{H0Eq}):
$$
\frac{1}{C'} \sup_M u - C' \leq \frac{1}{C}J(u) -C \leq d_1(0,u) 
\leq  C J(u) + C \leq C' \sup_M u + C'.
$$
\end{proposition}
\begin{proof}Let $u\in \H_0$. 
Equations \eqref{GreenEq} and \eqref{d1CharFormula} imply that 
$\frac{1}{C'}\sup_M u -C'\leq d_1(0,u)$.  Now, $u - \sup_M u \leq \min\{0,u\}$, so 
$u - \sup_M u \leq P(0,u)$. Thus, 
$-\sup_M u = \h{\rm AM}(u - \sup_M u) \leq \h{\rm AM}(P(0,u)).$ 
Combined with \eqref{Pythagorean}, 
$$
d_1(0,u)=-2\h{\rm AM}(P(0,u))\leq 2 \sup_M u.
$$
Finally, $J(u)$ and $\sup_M u$ are uniformly equivalent by
\eqref{JH0Eq} and \eqref{GreenEq}.
\end{proof}

Finally, we recall two crucial compactness results.
The first is a variant of a result of Berman et al \cite{bbegz}. 

\begin{theorem} 
\label{EntropyCompactnessThm}
Let $p>1$ and 
suppose $\mu = f \o^n$ is a probabilty measure with 
$f \in L^p(M)$.  
Suppose there exists $C>0$ such that $\{u_k\}_k \subset \E_1$ satisfies
$$
|\sup_M u_k| < C, \qq \int_M \log \frac{\o_{u_k}^n}{\mu}\o_{u_k}^n < C. 
$$
Then $\{ u_k\}$ contains a $d_1$-convergent subsequence. 
\end{theorem}

\bpf
According to Berman et al. \cite[Theorem 2.17]{bbegz}, 
$\{ u_k\}$ contains a subsequence $u_{j_k}$ converging 'in energy' to some $u \in \mathcal E_1$, i.e., $||u - u_{j_k}||_{L^1(M,\o^n)} \to 0$ and $\textup{AM}(u_{j_k}) \to \textup{AM}(u)$. According to Lemma \ref{L1Lemma}, this latter convergence is equivalent to $d_1$-convergence.
\epf

The second compactness result we recall
is an often used version of Zeriahi's generalization 
of Skoda's uniform integrability theorem \cite{ze}.

\begin{theorem} 
\label{ZeriahiCompactnessThm}
Consider the set
\beq
\lb{SkodaSetEq}
\big\{u\in\mathcal E_1\,:\, |\sup_M u|, \, |\h{\rm AM}(u)| \leq C\big\}.
\eeq
For any $p>0$ there exists $C'(C,p)>0$ such that for all $u$
belonging to \eqref{SkodaSetEq},
$$
\int_M e^{-pu}\o^n \leq C'.
$$
\end{theorem}

\bpf
The map $\vp \mapsto \h{\rm AM}(\vp)$ is 
upper semicontinuous (usc) 
with respect to the
$L^1(M,\on)$-topology, while the map
$\vp \mapsto \sup_M u$ is continuous
with respect to the $L^1(M,\on)$-topology. Thus,
the set \eqref{SkodaSetEq}
is compact with respect to the $L^1(M,\on)$-topology. 
According to \cite[Corollary 1.8]{gz} the elements of this set all have zero Lelong numbers. Hence, the requirements of \cite[Corollary 3.2]{ze} are satisfied finishing the proof.
\epf

\def\diff{\h{\rm diff}}

\subsection{Modified basic functionals arising from holomorphic vector fields}
\lb{ModifiedFunctSubSec}

Let $ \Aut_0(M,\JJJ) $ denote the connected component of the complex Lie group of automorphisms (biholomorphisms) of $(M,\JJJ)$ and denote by  $\aut(M,\JJJ)$ its Lie algebra of infinitesimal automorphisms composed
of real vector fields $X$ satisfying $\calL_X\JJJ=0$,
equivalently,
\beq
\lb{CxStrucEq}
\JJJ[X,Y]=[X,\JJJ Y], \q \forall\, X\in\aut(M,\JJJ), \; \forall\, Y\in\diff(M),
\eeq
where $\diff(M)$ denotes all smooth vector fields on $M$.
Thus $\aut(M,\JJJ)$ is a complex Lie algebra with complex structure $\JJJ$.

The automorphism group $\AutMJz$  acts on $\H$ by pullback:
\beq\lb{AutActionEq}
f.\eta:=f^\star\eta, \qq f\in\AutMJz, \q \eta\in\H.
\eeq
Given the one-to-one correspondence between $\mathcal H$ and $\mathcal H_0$, the group $\AutMJz$  also acts on $\H_0$. The action is described in the next lemma.
\blem
For $\vp \in \mathcal H_0$ and $f \in \AutMJz$ let $f.\vp\in\H_0$ be the unique element such that $f.\ovp=\o_{f.\vp}$.
Then,
\beq
\lb{factionAMEq}
f.\vp=f.0+\vp\circ f, \qq  f\in\AutMJz, \q \vp\in\H_0.
\eeq
\elem

\bpf
Note that \eqref{factionAMEq} is a \K potential for $f^\star\ovp$. Indeed,   $f\in\AutMJ$ implies that $f^\star\i\ddbar\vp=\i\ddbar \vp\circ f$.
That $\h{\rm AM}(f.0+\vp\circ f)=0$ follows from \eqref{AMDifferenceEq} as we have:
$$\h{\rm AM}(f . \vp)=\h{\rm AM}(f . \vp) - \h{\rm AM}(f.0)=\int_M \vp \circ f \sum_{j=0}^n f^\star \o^{n-j}\wedge f^\star \o_\vp^j = \h{\rm AM}(\vp)-\h{\rm AM}(0)=0.$$
\epf

\blem
\lb{dpIsomLemma}
The action of $\AutMJz$ on $\H_0$ is a $d_1$-isometry.
\elem

\bpf
From \eqref{factionAMEq}, 
$$
\frac{d}{dt}f.\vpt=\dot\vpt\circ f,
$$
for any smooth path $t \to \vp_t$ in $\H_0$. Thus, the $d_1$-length of $t \to f.\vpt$
is
$$
\int_0^1 V^{-1}\int_M|\dot\vp_t\circ f| f^\star \o_{\vp_t}^n dt
=
\int_0^1 V^{-1}\int_M|\dot\vp_t| \o_{\vp_t}^n dt,
$$
equal to the $d_1$-length of $\vpt$.
\epf

\blem
\label{ActionExtension}
The action of $\AutMJz$ on $\mathcal H_0$ has a unique $d_1$-isometric extension 
to the metric completion  $\overline{(\H_0,d_1)} =(\mathcal E_1 \cap \h{\rm AM}^{-1}(0),d_1)$.
\elem

\bpf Because $\AutMJz$ acts by $d_1$-isometries, each $f \in \AutMJz$ induces a $1$-Lipschitz continuous self-map of $\mathcal H_0$. By Lemma \ref{LipschitzExt}, such maps have a unique $1$-Lipschitz extension to the completion $\mathcal E_1 \cap \h{\rm AM}^{-1}(0)$ and the extension is additionally $d_1$-isometry. 
By density, the laws governing a group action have to be preserved as well.
\epf

Let $G\subset\AutMJz$ be a subgroup. Then the functional $J_G:\mathcal E_1 \cap \h{\rm AM}^{-1}(0)/G\ra \RR$ is introduced as
\beq
\lb{JGEq}
J_G(Gu):= \inf_{f \in G}  {J}(f.u).
\eeq
We have the follwing estimates:
\blem For $u \in \mathcal E_1 \cap \h{\rm AM}^{-1}(0)$ we have 
\lb{JGPropernessLemma}
\begin{equation}
\label{JGdGEqv}
 \frac{1}{C} J_G(Gu) -C \leq d_{1,G}(G0,Gu) 
\leq  C J_G(Gu) + C,
\end{equation}
where $d_{1,G}$ is the pseudometric of the quotient $\mathcal E_1 \cap \h{\rm AM}^{-1}(0)/G$.
\elem

\bpf
By Lemma \ref{dpIsomLemma}, 
$$
d_{1,G}(G0,Gu)=\inf_{f\in G} d_1(0,f.u).
$$
The result  now follows from Proposition \ref{Jproperness}.
\epf

The infinitesimal action of the Lie algebra $\autMJ$ on $\H$
associated to \eqref{AutActionEq}
naturally induces a vector field $\psi^X$ on $\H$ given by 
\beq
\lb{PsiXVectorFieldEq}
\eta\mapsto\psi^X_\eta\in\CinfM, \q \h{where $\calL_X\eta=\i\ddbar\psi^X_\eta$
 and $V^{-1}\int_Me^{\psi^X_\eta}\eta^n=1$.}
\eeq

Denote the 
$L^2(M,e^{\psi^X_\eta}\etan)$ inner product 
by $\langle\;\cdot\;,\;\cdot\;\rangle_{\psi_\eta^X}$.
The operator 
\beq
\lb{LetaXEq}
L_\eta^X:=\Delta_\eta+X
\eeq
is self-adjoint with respect to $\langle\;\cdot\;,\;\cdot\;\rangle_{\psi_\eta^X}$
whenever $X$ is the gradient vector field 
of $\psi^X_\eta$ with respect to $\eta$, i.e.,
\beq
\lb{XnablaEq}
X=\nabla \psi^X_{\eta}, 
\eeq
where the gradient is with respect
to $g_{\eta}$.
Observe that 
\beq
\lb{psiXvpEq}
\psi^X_{\ovp}=\psi^X_\o+X\vp
\eeq 
since it follows from the definition \eqref{PsiXVectorFieldEq} 
that \eqref{psiXvpEq} must hold
up to an additive constant, say, $C(\vp)$, but  
$$
\baeq
\frac d{dt} 
\int_M e^{\psi_{\o}^X+X(t\nu)}\o_{t\nu}^n
&=
\int_M L_{\o_{t\nu}}^X\nu e^{\psi_{\o}^X+X(t\nu)}\o_{t\nu}^n
\cr
&=
e^{-C(t\vp)}\langle L_{\o_{t\nu}}^X\nu,1\rangle_{\psi_{\o_{t\nu}}^X}=0.
\eaeq
$$
Thus, $C(t\vp)=0$ for each $t$.

\def\TJX{\h{$T_{\h{\notsosmall J} X}$}}

Let 
\beq
\lb{goJEq}
g_{\o}(\,\cdot,\,\cdot):=\o(\,\cdot,\JJJ\,\cdot).
\eeq
Consider the following hypothesis on a vector field 
$X\in\autMJ$:
\beq
\baeq
\lb{AssumpXEq}
&\h{$\bullet\;$ The closure $T=T(X)$ of the one-parameter subgroup  
generated by $\JJJ X$ is a subgroup of  
}
\cr
&\mskip22mu \h{of the isometry group of $(M,g_\o)$}. 
\eaeq
\eeq
As the compact group $T$ is the closure of a commutative subgroup of 
$\AutMJz$ it follows that $T$ is in fact a torus.
For any Lie subgroup $K$ of the isometry group of $(M,g_{\o})$
define the subspace
\beq
\lb{HKEq}
\H_\o^K:=\{\vp\in\H_\o\,:\, \vp \h{ is invariant under $K$}\},
\eeq
and similarly define $\H_0^K$.
According to Theorem \ref{d1CompletionThm}, the $d_1$-metric completion
of $\H_\o^K$ is
$$
\E_1^K:=\{u\in\E_1\,:\, u \h{ is invariant under $K$}\}.
$$

The following lemma is well-known \cite{Zhu,TZ}.
We include a proof since our notation is somewhat different than
in the original sources.

\blem
\lb{AMXDefLemma}
Suppose $X\in\aut(M,\JJJ)$ satisfies
\eqref{XnablaEq} for each $\eta\in\H^T$.
Let $\gamma:[0,1]\ra \H^T_\o$ denote a smooth path with
$\gamma(0)=0, \gamma(1)=\vp$.
The functional
$$
\h{\rm AM}_X(\vp):= V^{-1} \int_{[0,1]\times M}\dot\gamma(t) e^{\psi^X_{\o_{\gamma(t)}}}
\o_{\gamma(t)}^n\w dt
$$
is well-defined independently of the choice of $\gamma$.

\elem

\bpf
Indeed, the 1-form 
$\alpha_X:\vp\mapsto e^{\psi^X_{\ovp}}\ovpn$ on $\H^T_\o$
is closed since if $\nu,\mu\in T_\vp\H^T_\o$, 
and using \eqref{psiXvpEq},
$$
\frac d{dt}\Big|_{0} \alpha_X(\nu)|_{\vp+t\mu}
=
\int_M\nu L_{\ovp}^X \mu e^{\psi^X_{\ovp}}\ovpn
=
\frac d{dt}\Big|_{0} \alpha_X(\mu)|_{\vp+t\nu},
$$
as by our assumption on $X$ the operator $L_\ovp^X$ is self-adjoint
with respect to
$\langle\;\cdot\;,\;\cdot\;\rangle_{\psi_\ovp^X}$
as observed after \eqref{LetaXEq}.
\epf

\bremark
\lb{AMXRemark}
(i) As verified in the proof of Lemma \ref{AMXLemma} (i) below,
equation \eqref{XnablaEq}~follows~from~\eqref{AssumpXEq}.

\noindent
(ii)
We note in passing that from this lemma it follows that $\h{\rm AM}_X$ is monotone: $u \leq v$ implies  $\h{\rm AM}_X(u) \leq \h{\rm AM}_X(v)$. This parallels 
\eqref{AMDifferenceEq}.
\eremark

The next result follows using the arguments in the proofs
of Lemmas \ref{E1capH0Lemma} and \ref{ActionExtension}.

\blem
\lb{H0KCompletionLemma}
The metric completion of
$(\H_0^K,d_1)$ is 
$
\E_1^K\cap \h{\rm AM}^{-1}(0).
$
\elem

The Hodge decomposition implies that every $X\in\autMJ$
can be uniquely written as \cite{ga}
\beq
\lb{XDecompEq}
X=X_H+\nabla \psi^X_\o-\JJJ\nabla \psi^{\JJJsml X}_\o,
\eeq
where $\nabla$ is the gradient with respect to the Riemannian
metric \eqref{goJEq},
and $X_H$ is the $g_{\o}$-Riemannian dual of a 
$g_{\o}$-harmonic $1$-form.

\begin{lemma}
\lb{AMXLemma} 
Assume that $T=T(X)$ satifies \eqref{AssumpXEq}.
Assume also that 
the first Betti number of $M$ satisfies $b_1(M)=0$.

\noindent
(i) $\h{\rm AM}_X:\mathcal H_\o^T \to \Bbb R$ admits a unique $d_1$-continuous extension
to $\E_1^T$. 

\noindent
(ii) $\h{\rm AM}_X$ is usc with respect to the  $L^1(M,\o^n)$ topology of $\mathcal E_1^T$.
\end{lemma}

\begin{proof} 	
We claim that $\h{\rm AM}_X$ is $d_1$-Lipschitz continuous. 
We also claim that $X$ is a gradient field
with respect to all elements of $\H_\o^T$
so that $\h{\rm AM}_X$ is well-defined by Lemma \ref{AMXDefLemma}.
Lemma \ref{LipschitzExt} then gives (i). To prove these claims
observe first that as 
follows from \eqref{AssumpXEq},  the torus $T$ acts Hamiltonially on $(M,\o)$.
By \eqref{psiXvpEq} and \eqref{PsiXVectorFieldEq},
\beq
\lb{JXInvEq}
\psi^{\JJJsml X}_\ovp=\psi^{\JJJsml X}_\o=0
\eeq
for every $\vp\in\H^T_0$ since $\o$ is $T_{\JJJsml X}$-invariant. 
Since 
$b_1(M)=0$, $X_H=0$ in  \eqref{XDecompEq};
in particular, $X$ is a gradient vector field with 
respect to all $g_{\ovp}$ with $\vp\in\H^T_0$, i.e.,
\eqref{XnablaEq} holds.
Thus,  $\iota_{\JJJsml X}\ovp = d\psi^{X}_\ovp$.
In other words, $\psi^X_\ovp$ is the restriction of the moment map of the  $T$-action to the vector field ${\JJJ X}$ with respect to the symplectic form $\ovp$. 
By a general result of Atiyah--Guillemin--Sternberg the image of a moment map of a toric action is independent of the choice of symplectic form in the cohomology class, and therefore 
\beq
\lb{psiXEq}
|\psi^{X}_\ovp|\le C,
\eeq
for some uniform $C>0$ independent of $\vp\in\H_\o^X$. 
Thus, if $\gamma_t$ is any smooth path in $\H_0^X$ with endpoints
$u$ and $v$,
\begin{equation}\label{AMXdiff}
\baeq 
|\h{\rm AM}_X(u)-\h{\rm AM}_X(v)| 
& 
\le
\int_0^1 V^{-1}\int_M|\dot\gamma_t| e^{\psi^X_{\gamma_t}}\o_{\gamma_t}^n
\w dt
\cr
&
\leq 
C\int_0^1 
V^{-1}\int_M|\dot\gamma_t|\o_{\gamma_t}^n dt
\le 
Cl_1(\gamma).
\eaeq
\end{equation}
Taking the infimum over all such $\gamma$ we obtain
$$
|\h{\rm AM}_X(u)-\h{\rm AM}_X(v)| 
\le
C d_1(u,v),
$$
as desired.\\
\noindent (ii) This is a consequence of 
the monotonicity of $\h{\rm AM}_X$ 
(Remark \ref{AMXRemark} (ii))
and 
is proved in \cite[Proposition 2.15]{bwn}.
\end{proof}

\subsection{Condition \ref{a2}: action functionals and their
lsc extensions}
\lb{ActionSubSec}

Each set of canonical \K metrics we consider in this article
can be defined as the minimizers of an appropriate action functional
over a space $\calR$ of `regular' potentials.
It is crucial for us, however, to understand the greatest $d_1$-lsc
extensions of these functionals to the corresponding $d_1$-metric
completion $\overline{\calR}$.
This is the main goal of the present rather long subsection.
We emphasize that it is crucial for us to obtain explicit formulas
for the lsc extensions of our functionals in order to be able to 
apply existing results concerning regularity
of minimizers over $\overline{\calR}$. That is, if we only had
an abstract, but not explicit, extension, this would not have
allowed us to conclude property \ref{p3}.

\subsubsection{\KE (edge) metrics}

For an introduction to K\"ahler edge geometry 
we refer to the expository article \cite{R14}. 
Let $D\subset M$ denote a smooth divisor, and let $\beta\in(0,1]$.
Suppose that $h$ is a smooth Hermitian
metric on the line bundle $L_{D}$ associated to $D$, 
and that $s$ is a global holomorphic section of
$L_{D}$, so that $D=s^{-1}(0)$. 
Let $\o$ be a smooth \K metric on $M$ and define
\beq
\label{oEq}
\o_c:=\o+c\i\ddbar(|s|^2_{h})^{\be},
\eeq
For $c>0$ small enough, $\o_c$ is a 
smooth \K metric away from $D$ \cite[Lemma 2.2]{JMR}.
Fix such a $c$ once and for all.

\bdefin
\lb{KEEDef}
A K\"ahler current $\eta$ on $M$ is called
a {\it \K edge form of angle $\beta$} if it is a smooth K\"ahler form on
$M \setminus D$ and satisfies
$\o_c/C\le \eta \le C\o_c$ for some constant $C=C(\eta)>0$.
\edefin

The set of \K edge potentials (of angle $\beta$) is denoted by
$$
\mathcal H^\beta_\o:= 
\{ u \in \PSH(M,\o) :\ \o_u \textup{ is a K\"ahler edge form of angle }\beta\}.
$$
Set also,
$$
\mathcal H^\beta_0:=\mathcal H^\beta \cap \h{\rm AM}^{-1}(0).
$$

There are plenty of K\"ahler edge metrics, according to the next result.

\begin{lemma}
\label{HbetaDense}  
The $d_1$-metric completion of $\mathcal H^\beta_\o$ is $\mathcal E_1$.
\end{lemma}
\begin{proof}
Let $u \in \mathcal H_\o$, and write
$$
u_k:= u + \frac{1}{k}(|s|^{2}_h)^\beta.
$$
For $k$ large enough $u_k \in \mathcal H^\beta_\o$ \cite[Lemma 2.2]{JMR}. 
As $u_k$ is decreasing pointwise to $u$, Lemma \ref{L1Lemma} implies that $d_1(u_k,u) \to 0$. As $\mathcal H_\o$ is dense in $\mathcal E_1$, it follows that the metric completion of $\mathcal H_\o^\beta$ is $\mathcal E_1$.
\end{proof}

We turn our attention to \KE edge metrics. Suppose that
\beq \label{ConicCohomCond}
c_1(M)-(1-\be)[D]=\mu[\o_0]/2\pi.
\eeq
\beq\label{foEq}
\i\ddbar f_{\eta,\beta} =\Ric\,\eta-2\pi(1-\be)[D]-\mu\eta,
\q
\int_M e^{f_{\eta,\beta}}\etan=\int_M \etan.
\eeq
A \K edge form $\eta$ is called \KE edge (KEE) 
when $f_{\eta,\beta}=0$.
We assume from now and on that $\mu>0$ 
since the existence problem in the case $\mu\le 0$ is
settled in \cite{JMR}. In fact after rescaling the K\"ahler class 
we assume that $\mu=1$.

We record the following estimate.
\blem
\lb{RicciPotentialLemma}
Let $\o$ be a smooth \K form on $M$ and let
$\eta$ be either a smooth \K form or a \K edge
form of angle $\be$. Then,
$e^{f_{\eta,\beta}} \in L^p(M,\on)$ for some $p>1$.
\elem

\bpf
When $\eta$ is smooth this 
is a direct consequence of the Poincar\'e--Lelong formula.
When $\eta$ is a \K edge form of angle $\be$ then
$f_{\eta,\be}$ is actually continuous \cite[\S4]{JMR}.
\epf

The Berger--Moser--Ding energy
(or Ding functional for short) \cite{Ding1988} is defined by
\beq
\label{FbetaEq}
F^\be(\vp)=F^\be(\ovp)
:=-\h{\rm AM}(\vp)-\log \frac{1}{ V}\int_M e^{f_{\o,\beta}-\vp}{\o}^n, \ \ \vp \in \mathcal H^\beta.
\eeq
The Mabuchi K-energy $E^\be:\H^\be_\o\ra\RR$ is closely related and defined by
\cite[(5.27)]{R14},\cite{Mabuchi87}, 
\beq
\label{EbetaEq}
E^\beta(\vp) 
:=
\Ent(e^{f_{\o,\beta}} \o^n,\o^n_\vp)
-  \h{\rm AM}(\vp)
+ \V\int_M \vp\ovp^n.
\eeq
Observe that: (i) \eqref{EbetaEq} differs from \cite[(5.27)]{R14}
by a constant equal to $\V\int f_{\o,\beta}\on$, (ii)
the term $-(I-J)(\o,\ovp)$ there equals the last two terms in 
\eqref{EbetaEq} by \eqref{IJIEq}. 
Here,
\beq
\Ent(\nu,\chi)=
\frac{1}{V}\int_M\log\frac{\chi}{\nu}\chi,
\eeq
is the entropy of the measure $\chi$ with respect to the measure $\nu$. The following formula relating $F^\beta$ and $E^\beta$ was established by Ding--Tian \cite{dt}:
\beq
\label{DingTianEq}
E^\be(\vp)=  F^\be(\vp)-\frac1V\int 
f_{\ovp,\beta}\ovpn.
\eeq
Jensen's inequality gives that $V^{-1}\int_M f_{\o_{\vp,\beta}} \o_\vp^n\leq \log V^{-1}\int_M e^{f_{\o_\vp,\beta}}\o_\vp^n = 0$, hence there exists $C>0$ such that
\beq\label{EbetaFbetaIneq}
E^\beta(\vp) \geq F^\beta(\vp) 
\eeq
The critical points in $\H^\be_\o$ of both $F^\be$ and $E^\be$  are precisely
\KE (edge) potentials. 
The next result verifies condition \ref{a2} for these functionals.

\bprop
\label{FbetaExt}
Formula \eqref{FbetaEq} gives the unique $d_1$-continuous extension of 
$F^\be:\mathcal H^\beta \to \Bbb R$ to a functional on $\E_1$.
\eprop

\begin{proof} 
Lemma \ref{L1Lemma} (i) provides a $d_1$-continuous extension of $\h{\rm AM}$.
To deal with the other term we prove the following lemma.

\begin{lemma} 
\label{ExtLemma} 
Let $e^a \in L^{p}(M,\on)$ for some $p>1$.
Then the functional defined on $\mathcal E_1$ by
$$
\tilde F(\vp) = \int_M e^{a-\vp}\o^n,
$$ 
is $d_1$-continuous.
\end{lemma}

\begin{proof} 
Let $u_k,u \in \mathcal E_1$ be such that $d_1(u_k,u) \to 0$. We have to argue that $\tilde F(u_k) \to \tilde F(u)$. As $d_1(0,u_j)$ is bounded it follows 
from Lemma \ref{L1Lemma} (i) and \eqref{GreenEq}
that for some $C>0$ independent of $j$,
\beq
\lb{supAMEq}
|\sup_M u_j|, \ |\h{\rm AM}(u_k)| \leq C.
\eeq
Hence, Theorem \ref{ZeriahiCompactnessThm} implies there exists $C'(s)>0$ 
independent of $j$ such that
\begin{equation}
\label{Zeriahiest1}\int_M e^{-su_j}\o^n \leq C', \ j \in \Bbb N,
\end{equation}
for all $s \geq 1$. Since $x\mapsto e^x$ is convex,
\beq
\lb{convExpEq}
|e^a - e^b| \leq |a-b|\max\{e^a,e^b\}\le |a-b|(e^a+e^b). 
\eeq
Suppose that $1/q + 1/p + 1/s=1$. 
Then, \eqref{convExpEq} and H\"older's inequality imply that
\begin{flalign}\label{expIntegralEst}
\Big|
\int_M  
\Big(e^{a  - u_{k}} 
- e^{a - u}\Big)\o^n
\Big| 
&\leq 
\int_M |u_{k}-u|(e^{a- u_{k}}+e^{a- u})\o^n \\
&\leq 
||u_{k}-u||_{L^q(M,\on)}
||e^a||_{L^p(M,\on)}
(||e^{-u_{k}}||_{L^s(M,\on)}+||e^{-u_{k}}||_{L^s(M,\on)}).\nonumber
\end{flalign}
The first term of this last expression converges to zero
since by Lemma \ref{L1Lemma} (i), $\int_M |u_{k} -u|\o^n \to 0$ while
all $L^s$ topologies are equivalent on $\PSH(M,\o)$.
The second term is bounded by hypothesis, while 
the third is bounded by \eqref{Zeriahiest1}.
\epf

Thus, the proposition follows from the lemma by setting 
$a=f_{\o,\beta}$, Lemma \ref{RicciPotentialLemma}
and the $d_1$-density of 
$\mathcal H^\beta$ in $\mathcal E_1$ proved in Lemma \ref{HbetaDense}.
\epf

\bprop
\label{EbetaExt}
Formula \eqref{EbetaEq} gives the greatest $d_1$-lsc extension of $E^\beta:\mathcal H^\beta \to \Bbb R$ to a functional on $\E_1$.
\eprop

A more precise version of this result is contained in the forthcoming paper \cite{bdl}, but for completeness we give a proof here:

\bpf
The last two summands of \eqref{EbetaEq} can be rewritten as
\beq
\lb{last2summandsEq}
n\h{\rm AM}(\vp)-(n+1)\Big[\h{\rm AM}(\vp)-\frac{\V}{n+1}\int_M\vp\ovpn\Big]
=
n\h{\rm AM}(\vp)
- \V \sum_{j=0}^{n-1}\int_M \vp \wedge \ovp^{j} \wedge \o^{n-j}. 
\eeq
Combining Lemma \ref{L1Lemma} (i)
and the following lemma (with $\alpha=\o$) it follows that \eqref{last2summandsEq}
is $d_1$-continuous.

\bremark
At this stage, we could have also simply used \eqref{EbetaEq}
and proved
 that $\vp\mapsto\int_M\vp\ovpn$ is $d_1$-continuous.
We choose, however, to use the somewhat more complicated formula
\eqref{last2summandsEq} since then essentially the same arguments
allow us to deal with the K-energy on a general \K class (see 
Proposition \ref{EcscExtProp}). 
\eremark

\blem
\lb{ExtLemma2}
Suppose $\alpha$ is a smooth closed $(1,1)$-form on $M$. The functional defined on $\mathcal E_1$ by
$$
\tilde E(\vp):= 
\V \sum_{j=0}^{n-1}\int_M \vp \alpha \wedge \ovp^{j} \wedge \o^{n-1-j}
$$
is $d_1$-continuous and bounded on $d_1$-bounded subsets of $\mathcal E_1$.
\elem

\bpf
As in the proof of Lemma \ref{E1capH0Lemma} (i), $\tilde E$ indeed
makes sense on $\E_1$. Now, let $u_k,u \in \mathcal E_1$ 
be such that $d_1(u_k,u) \to 0$. 
An argument similar to that yielding \eqref{AMDifferenceEq} shows that
$$
\tilde E(u) -\tilde E(u_k)= 
\V\int_M \sum_{j=0}^{n-1} (u-u_k) \alpha \wedge \o_{u}^{j} \wedge \o_{u_k}^{n-1-j}.
$$
It is clear that for some $D>0$ we have $-D\o \leq \alpha \leq D\o$. Thus, observing that
$\o_{(u+u_k)/{4}}=\o/2+\o_v/4+\o_u/4$,
\beq\label{integralest}
|\tilde E(u) - \tilde E(u_k)| \leq C \int_M |u-u_k| 
\o_{(u+u_k)/{4}}^n.
\eeq
By \cite[Corollary 5.7]{da4} and its proof, for each $R>0$ there exists 
$f_R:\Bbb R \to \Bbb R$ continuous with $f_R(0)=0$ such that
\beq \label{da4CorollaryEst}
\int_M |v-w|\o_h^n \leq f_R(d_1(v,w)),
\eeq
for any $v,w,h \in \mathcal E_1\cap\{\vp\,:\,d_1(0,\vp)\le R\}$.
Using this, to show that \eqref{integralest} converges to $0$, it is enough to argue that $d_1(0,(u+u_k)/4)$ is uniformly bounded. For this we use  \cite[Lemma 5.3]{da4} that says that there exists $C>1$ such that $d_1(v,(w+v)/2) \leq C d_1(v,w)$ for any $v,w \in \mathcal E_1$. Using this several times and the triangle inequality,
\begin{flalign*}
d_1(0 ,(u+u_k)/4) \leq & Cd_1(0,(u+u_k)/2) \leq C(d_1(0,u)  + d_1(u,(u+u_k)/2)) \\
\leq & C^2(d_1(0,u) + d_1(u,u_k)).
\end{flalign*}
This last term is uniformly bounded, showing that $\tilde E$ is $d_1$-continuous. Also, by \eqref{da4CorollaryEst} it follows that $\tilde E$ is bounded on $d_1$-bounded subsets of $\mathcal E_1$.
\end{proof}

To prove Proposition \ref{EbetaExt} it thus remains to deal with the entropy term.

First, we claim this term can be extended to $\E_1$ in a $d_1$-lsc fashion.
Indeed, $d_1$-convergence implies 
weak convergence of volume measures  \cite[Theorem 5 (i)]{da4}. 
By the display following (5.29) in \cite{R14},
the map $\chi \to \Ent(\mu,\chi)$ is a supremum of
continuous maps with respect to 
weak convergence of measures, hence is lsc with respect to this
same convergence. 
These last two statements together imply that 
the map $u \to \Ent(e^{f_\o,\beta}\o^n,\o_u^n)$ is $d_1$-lsc,
as claimed.

Second, the following result shows that thus extended to $\E_1$, 
the entropy is the greatest $d_1$-lsc extension of its restriction
to $\H^\be_\o$. The result is due to a forthcoming paper  where much more precise results are proved \cite{bdl}:

\begin{lemma} Given $u \in \mathcal E_1$, there exists $u_k \in \mathcal H^\beta$ such that $d_1(u_k,u) \to 0$ and $$\Ent(e^{f_{\o,\beta}}\o^n,\o_{u_k}^n) \to \Ent(e^{f_{\o,\beta}}\o^n,\o_{u}^n).$$
\end{lemma}

\begin{proof} If $\Ent(e^{f_{\o,\beta}}\o^n,\o_{u}^n)=\infty$, than any sequence $u_k \in \mathcal H^\beta$ with $d_1(u_k,u) \to 0$ satisfies the requirements, as follows from the $d_1$-lower semi-continuity of the entropy. We can suppose that $Ent(e^{f_{\o,\beta}}\o^n,\o_{u}^n)$ is finite.

Let $g = \o_{u}^n/\o^n \geq 0$ be the density function of $\o_u^n$. We argue that there exists positive functions $g_k \in C^\infty(X)$ such that $|g-g_k|_{L^1} \to 0$ and 
$$\int_M g_k \log\frac{g_k}{e^{f_{\o,\beta}}}\o^n  \to \int_M g \log\frac{g}{e^{f_{\o,\beta}}}\o^n=\Ent(e^{f_{\o,\beta}}\o^n,\o_u^n).$$
First introduce $h_k = \min \{k,g\}, \ k \in \NN$. As $\phi(t) =t\log(t), \ t \geq 0$ is bounded from below by $-e^{-1}$ and increasing for $t\geq 1$, it follows that 
$$
-e^{-1} e^{f_{\o,\beta}} \leq h_k \log\frac{h_k}{e^{f_{\o,\beta}}} 
\leq \max \{ 0, g \log\frac{g}{e^{f_{\o,\beta}}}\}. $$ 
Clearly  $|h_k -g|_{L^1} \to 0$, and the dominated  convergence theorem gives that
\beq\label{Entconv1}
\int_M h_k \log\frac{h_k}{e^{f_{\o,\beta}}}\o^n  \to \int_M g \log\frac{g}{e^{f_{\o,\beta}}}\o^n=\Ent(e^{f_{\o,\beta}}\o^n,\o_u^n).
\eeq
Using the density of $C^\infty(M)$ in $L^1(M)$, by another application of the  dominated convergence theorem, we can find a positive sequence $g_k \in C^\infty(X)$  such that $|g_k-h_k|_{L^1} \leq 1/k$ and 
\beq \label{Entconv2}
\Big|\int_M h_k \log\frac{h_k}{e^{f_{\o,\beta}}}\o^n  - \int_M g_k \log\frac{g_k}{e^{f_{\o,\beta}}}\o^n\Big| \leq \frac{1}{k}.
\eeq
Using the Calabi-Yau theorem we find potentials $v_k \in \mathcal H_\o$ with $\sup_M v_k=0$ and $\o_{v_k}^n = g_k \o^n/\int_M g_k \o^n$. Proposition \ref{Ebetacompact} now guarantees that after possibly passing to a subsequence $d_1(v_k,h) \to 0$ for some $h \in \mathcal E_1(X)$. But \cite[Theorem 5 (i)]{da4} implies the equality 
of measures $\o_h^n = \o^n_u$.
 Finally the uniqueness theorem 
\cite[Theorem B]{gz} gives that in fact $h$ and $u$ 
can differ by at most a constant. Hence, after possibly adding a constant, 
we can suppose that $d_1(v_k,u) \to 0$.

The last step is to perturb $v_k$ slightly to obtain the conical metrics $u_k = v_k + \varepsilon_k |s|_h^{2\beta} \in \mathcal H^\beta$, where $s$ is the section of $L_D$, with vanishing locus $D$. The argument of Lemma \ref{HbetaDense} gives that for small enough $\varepsilon_k>0$ one has
$$d_1(u_k,u) \to 0.$$
Using \eqref{Entconv1} and \eqref{Entconv2} a  basic calculation gives that after possibly shrinking $\varepsilon_k>0$ further we obtain
$$
\Ent(e^{f{\o_\beta}}\o^n,\o^n_{u_k}) \to \Ent(e^{f{\o_\beta}}\o^n,\o^n_{u}),
$$
as desired.
\end{proof}

Since the last two terms in the formula for $E^\be$ are $d_1$-continuous,
as already noted, the previous lemma implies that our extension
is the largest $d_1$-lsc extension of its restriction to $\H_\o^\be$.
This concludes the proof of Proposition \ref{EbetaExt}.
\epf

\def\V{V^{-1}}
\def\fovpt{f_{\ovpt}}

\subsubsection{\KR solitons}

Suppose that $(M,\JJJ)$ is Fano and let $X\in\aut(M,\JJJ)$. 
Assume that $X$ satisfies \eqref{AssumpXEq}.
Recall from the proof of Lemma (i) that this implies
$X$ is a gradient vector field with respect to metrics in $\H^T$.
We say that  $\eta \in \mathcal H^T$ is a \KRS associated to $X$ if 
\begin{equation}
\label{KRsolitonEq}
\Ric\,\eta =\eta + \calL_X\eta.
\end{equation}
Following Tian-Zhu, we define the modified Ding functional and K-energy functional on $\H^T_\o$ as follows:
\begin{equation}
\label{ModifFEq}
F^X(\vp):=-\h{\rm AM}_X(\vp)-\log \V\int_M e^{f_\o - \vp}\o^n,
\end{equation}
\beq
\label{ModifEEq}
E^{X}(\vp):= F^X(\vp) + 
\V\int_M (f_\o - \psi^X_\o)e^{\psi^X_\o}\o^n 
- 
\V\int_M (f_{\o_\vp} - \psi^X_{\o_\vp})e^{\psi^X_{\o_\vp}}\o_\vp^n.
\eeq
As in \eqref{EbetaFbetaIneq}, it follows that
\beq
\label{EmodifFmodifIneq}
E^{X}(\vp) \geq  F^X(\vp) -C.
\eeq
The critical points of these functionals are K\"ahler-Ricci solitons. Both $F^X$ and $E^X$ are defined for smooth potentials, but as it turns out their definition can be extended to $\mathcal E^X_1$. 
In this article we will only need this for $F^X$. 

\bprop 
\label{FModifExt} Formula \eqref{ModifFEq} gives a unique $d_1$-continuous 
extension of $F^X:\mathcal H^X_\o \to \Bbb R$ to a functional on $\E^X_1$.
\eprop

\begin{proof} Lemma \ref{AMXLemma} gives that the  map $t \to \h{\rm AM}_X(u)$ is $d_1$-continuous. Lemma \ref{ExtLemma} with 
$a=f_\o$ give that the map $u \to  \log V^{-1 }\int_M e^{f_\o - u}\o^n$  
is also $d_1$-continuous. Uniqueness of the extension now 
follows from the density of $\mathcal H^X_\o$ in $\mathcal E_1^X$.
\end{proof}

\subsubsection{Constant scalar curvature metrics}

On a general K\"ahler manifold $(M,\o)$ 
Mabuchi's K-energy equals 
\begin{equation}
\label{Kendef}
E(\vp):= 
\Ent(\o^n,\o^n_\vp)
+  
n\h{\rm AM}(\vp) 
- 
\frac{1}{V}\sum_{j=0}^{n-1}\int_M \vp \Ric \o \wedge \ovp^{j} \wedge \o^{n-1-j}.
\eeq
Recall that by \eqref{last2summandsEq} 
this expression is equal to the one in \eqref{EbetaEq} when $\beta=1$. The advantage of this definition over the one in \eqref{EbetaEq} is that this makes sense independently of the condition \eqref{ConicCohomCond}.

The critical points of $E$ over $\H_0$ are metrics of constant scalar curvature. We have the following result, 
first obtained in \cite{bdl}, whose proof follows exactly the same line as that of Proposition \ref{EbetaExt}, the only difference being that in Lemma \ref{ExtLemma2} we now take $\alpha = \Ric \o$.
\bprop
\lb{EcscExtProp}
Formula \eqref{Kendef} 
gives the greatest $d_1$-lsc extension of 
$E:\mathcal H_\o \to \Bbb R \cup \{ \infty \}$ to a functional on $\E_1$.
\eprop

\subsection{Property \ref{p2}: compactness of minimizing sequences}

We turn to the compactness condition \ref{p2}. Most of the results that  we present here have already been obtained in the works \cite{bbgz,bbegz,bwn} in a different context. To fit well with the metric-geometric viewpoint on the space of K\"ahler metrics we employ here, we see it adequate to present a very detailed account of all theorems, often following the ideas of \cite{bbgz,bbegz,bwn}.

\subsubsection{\KE (edge) metrics}

\begin{proposition} 
\label{Fbetacompact} 
Suppose $(M,\JJJ,D, \o)$ satisfies \eqref{ConicCohomCond} with $\mu=1$.
Let $\calR=\H^\be_\o$, so $\overline{(\calR,d_1)}=\E_1$.
Then $F^\be$ given by Proposition \ref{FbetaExt} satifies property \ref{p2}.
\end{proposition}

\begin{proof} 
{\it Step 1.} In this step we construct a candidate minimizer $u\in \E_1$.
Suppose that $\{u_j\}_{j \in \Bbb N} \subset \mathcal \E^X_1$ 
satisfies 
$$
\lim_jF^\beta(u_j)= \inf_{\mathcal \E^X_1} F^\beta, \qq d_1(0,u_j) \leq C.
$$ 
Thus, as in the proof of Lemma \ref{ExtLemma}, the estimates
\eqref{supAMEq} hold, in particular 
$|\sup_M u_j| \leq C$, so there exists 
$j_k \to \infty$ and $u \in \textup{PSH}(M,\o)$ such that 
$\int_M |u_{j_k} -u|\o^n \to 0$ \cite[Proposiion I.4.21]{dembook}.  

Now, Lemma \ref{AMXLemma} (ii) (with $X=0$) gives that $\h{\rm AM}$ is 
upper semicontinuous with respect to the weak $L^1(M,\on)$ topology.
This together with $|\h{\rm AM}(u_{j_k})| \leq C$ yields,
\beq
\lb{AMuscEq}
-C \leq \limsup_{k} \h{\rm AM}(u_{j_k})\leq \h{\rm AM}(u) \leq \limsup_{k} \sup_M {u_{j_k}} \leq C,
\eeq
hence $u \in \mathcal E_1$ by Lemma \ref{AMfinite}. 
\smallskip

\noindent
{\it Step 2.} In this step we show $u$ is actually a minimizer.
Following the ideas giving estimate \eqref{expIntegralEst} we arrive at
\begin{flalign}
\Big|
\int_M  
\Big(e^{f_{\o,\beta}  - u_{k}} 
- &e^{f_{\o,\beta} - u}\Big)\o^n
\Big| 
\leq 
\int_M |u_{k}-u|(e^{f_{\o,\beta}- u_{k}}+e^{f_{\o,\beta}- u})\o^n 
\cr
&
\leq 
||u_{k}-u||_{L^q(M,\on)}
||e^{f_{\o,\beta}}||_{L^p(M,\on)}
(||e^{-u_{k}}||_{L^s(M,\on)}+||e^{-u_{k}}||_{L^s(M,\on)}),
\nonumber
\end{flalign}
where $1/q + 1/p + 1/s=1$. The first term of this last expression converges to zero as all $L^q$ topologies are equivalent on $\PSH(M,\o)$. The second term is bounded by Lemma \ref{RicciPotentialLemma}, while the third is bounded by \eqref{Zeriahiest1}. All this gives
\begin{equation}
\label{liminfAM}
\lim_{j_k}F^\beta(u_{j_k}) 
\geq 
- 
\limsup_{j_k}\h{\rm AM}(u_{j_k}) 
-\log\Big( \V\int_M e^{f_{\o,\be}-u}\o^n\Big) 
\geq F^\beta(u).
\end{equation}
As $j_k \to F^\beta(u_{j_k})$ minimizes $F^\beta$, 
it follows that the last inequality must be an equality.
Thus,  $u$ minimizes $F^\beta$.

\smallskip
\noindent{\it Step 3.} Here, we show that there is a subsequence $d_1$-converging to $u$.
Since equality holds in \eqref{liminfAM},
$\limsup_k \h{\rm AM}(u_{j_k})=\h{\rm AM}(u)$. 
Thus, after possibly passing to a further subsequence, 
$\lim_m \h{\rm AM}(u_{j_{k_m}})=\h{\rm AM}(u)$. 
This together with $||u_{j_{k_m}}-u||_{L^1(M,\on)} \to 0$ 
and Lemma \ref{L1Lemma} (i) gives that $d_1({j_{k_m}},u)\to 0$.
\end{proof}

\begin{proposition} 
\label{Ebetacompact}
Suppose $(M,\JJJ,D, \o)$ satisfies \eqref{ConicCohomCond} with $\mu=1$.
Let $\calR=\H^\be_\o$, so $\overline{(\calR,d_1)}=\E_1$.
Then $E^\be$ given by Proposition \ref{EbetaExt} satifies property \ref{p2}. 
\end{proposition}

\begin{proof}  
According to Lemma \ref{ExtLemma2}  and Lemma \ref{AMXLemma} (i) (with $X=0$), 
the second and third summands in \eqref{EbetaEq} are 
controlled by the metric $d_1$. 
Thus, as $E^\beta(u_j)$ and $d(0,u_j)$ are 
uniformly bounded, this implies that both 
$\Ent(f_{\o,\beta}\o^n,\o_{u_j}^n)$ and $|\sup_M u_j|$ remain uniformly bounded. 
Theorem \ref{EntropyCompactnessThm} thus gives a subsequence 
$u_{j_k}$ and $u \in \mathcal E_1$ such that $d_1(u_{j_k},u) \to 0$. 
The $d_1$-lower semicontinuity of $E^\beta$ 
(Proposition \ref{EbetaExt})
then implies that $u$ minimizes $E^\beta$ and 
$\lim_k E^\beta(u_{j_k}) = E^\beta(u)$.
\end{proof}

\subsubsection{\KR solitons}

\begin{proposition} 
\label{FModifCompact} 
Suppose $(M,\JJJ,\o)$ is Fano, $X\in\autMJ$ satisfies \eqref{AssumpXEq}.
Let $\calR=\H^X_\o$, so $\overline{(\calR,d_1)}=\E^X_1$.
Then $F^X$ given by Proposition \ref{FModifExt} satifies property \ref{p2}.
\end{proposition}

\begin{proof} 
The proof is similar to that of Proposition \ref{Fbetacompact}, 
with a slight twist at the end.

{\it Step 1.} This is identical to Step 1 in the 
proof of Proposition \ref{Fbetacompact}, yielding a $u\in\E_1^X$
satisfying \eqref{AMuscEq}.

\smallskip
\noindent
{\it Step 2.}
This is again identical to Step 2 in the 
proof of Proposition \ref{Fbetacompact}, this time yielding
\begin{equation}
\label{liminfAMX}
\lim_{j_k}F^X(u_{j_k}) 
\geq - \limsup_{j_k}\h{\rm AM}_X(u_{j_k}) 
-\log\Big( \V\int_M e^{f-u}\o^n\Big) \geq F^X(u).
\end{equation}
As $j_k \to F^X(u_{j_k})$ minimizes $F^X$, it follows that the  inequalities are actually equalities, so  
$u$ minimizes $F^X$ and $\limsup_k \h{\rm AM}_X(u_{j_k})=\h{\rm AM}_X(u)$. Consequently, after possibly passing to a subsequence 
\beq\label{AMXlimitEq}
\lim_k \h{\rm AM}_X(u_{j_k})=\h{\rm AM}_X(u).
\eeq

\smallskip
\noindent
{\it Step 3.}
We argue now that $d_1(u_{j_k},u) \to 0$. 
By Lemma \ref{L1Lemma} we only need to show that 
$\h{\rm AM}(u_{j_k}) \to \h{\rm AM}(u)$. 
For this we introduce an auxiliary sequence  
$$
v_{j_k}:= 
\max\{u,u_{j_k}\} \in \mathcal E_1^X.
$$
Observe that
\beq\lb{SupvjkEq}
|\sup_M v_{j_k}|\le C,
\eeq
since this holds for $u_{j}$ by \eqref{supAMEq} (this is part of Step 1).
Similarly, since $v_{j_k}\ge u_{j_k}$, \eqref{AMMonEq} and \eqref{supAMEq} imply
that $\h{\rm AM}(v_{j_k})\ge C$. Since \eqref{SupvjkEq} and the definition
of AM give that $\h{\rm AM}(v_{j_k})\le C$, we have
$$
|\h{\rm AM}(v_{j_k})| \leq C.
$$
Furthermore, $||v_{j_k}-u||_{L^1(M,\on)} \to 0$. 
Thus, we may apply Step 2 to the sequence $\{v_{j_k}\}$ to obtain
\eqref{liminfAMX} for $v_{j_k}$.
By monotonicity of $\h{\rm AM}_X$ we have 
$-\h{\rm AM}_X(v_{j_k}) 
\leq -\h{\rm AM}_X(u_{j_k})$, 
hence  $v_{j_k}$ is also $F^X$-minimizing. 
In the same manner \eqref{AMXlimitEq} was derived one then has
\beq\label{vAMXlimitEq}
\lim_k \h{\rm AM}_X(v_{j_k})=\h{\rm AM}_X(u).
\eeq
Examining 
\eqref{AMXdiff} and recalling 
\eqref{psiXEq}, 
one can show the existence of $C>1$ 
such that  for any $w, v \in \mathcal E_1$ {\it with} $w \leq v$ one has
\begin{equation}\label{AMAMXeqv}
0\leq \frac{1}{C}(\h{\rm AM}(v) - \h{\rm AM}(w)) \leq \h{\rm AM}_X(v) - \h{\rm AM}_X(w) \leq C(\h{\rm AM}(v) - \h{\rm AM}(w)).
\end{equation}
Using that $u \leq v_{j_k}$, $u_{j_k} \leq v_{j_k}$ equations \eqref{AMXlimitEq}, \eqref{vAMXlimitEq} and \eqref{AMAMXeqv} we obtain that 
$$
\lim_k \h{\rm AM}(u_{j_k})=\lim_k \h{\rm AM}(v_{j_k})=\h{\rm AM}(u),
$$
as desired.
\end{proof}

\subsubsection{Constant scalar curvature metrics}

The proof of the following result is the same as that of Proposition \ref{Ebetacompact} and it already appears in \cite{dh}.

\begin{proposition} 
\label{EcscKcompact}
Let $\o$ be an arbitrary \K metric on $(M,\JJJ)$.
Let $\calR=\H_\o$, so $\overline{(\calR,d_1)}=\E_1$.
Then $E$ given by Proposition \ref{EcscExtProp} satifies property \ref{p2}. 
\end{proposition}

\subsection{Property \ref{p3}: regularity of minimizers}
\label{P3SubSec}

The regularity condition \ref{p3} for $F^\be$ and $E^\be$ follows after combining together the result of Berman \cite[Theorem 1.1]{brm1} 
and the regularity theorems 
\cite[Corollary 6.9]{ko},\cite{JMR},\cite{gp}
(see \cite{R14} for more details).

\begin{theorem} 
\label{BRM1Thm}
Suppose $(M,\JJJ,D, \o)$ satisfies \eqref{ConicCohomCond} with $\mu=1$ 
and $\vp \in \E_1$. The following are equivalent:
\noindent (i) $\ovp$ is a \KE edge metric,
\noindent (ii) $\vp$ minimizes $F^\beta$,
\noindent (iii) $\vp$ minimizes $E^\beta$.
\end{theorem}

The regularity condition \ref{p3} for $F^X$ and $E^X$ is
due to the following result of Berman--Witt Nystr\"om \cite[Theorem 3.3]{bwn}.

\begin{theorem} 
\label{BWNRThm}
Suppose $(M,\JJJ,\o)$ is Fano, and that $X\in\autMJ$
satisfies \eqref{AssumpXEq}. Let $\vp \in \E^X_1$. The following are equivalent:
\noindent (i) $\ovp$ is a \KR soliton,
\noindent (ii) $\vp$ minimizes $F^X$,
\noindent (iii) $\vp$ minimizes $E^X$.

\end{theorem}

\section{Property \ref{p6} via a partial Cartan decomposition}
\lb{CartanSec}

The purpose of this section is to give explicit criteria 
that imply  condition \ref{p6} in all cases of interest for us. The main 
result is Proposition \ref{p6Prop}, while the main
technical ingredient is Proposition \ref{PartialCartanProp}.

\subsection{Partial Cartan decomposition}

Recall the following form of the classical Cartan decomposition
\cite[Proposition 32.1, Remark 31.1]{Bump}.

\bthm
\lb{CartanThm}
Let $S$ be a compact connected semisimple Lie group. Denote
by $(S^\CC,\JJJ)$ the complexification of $S$, namely the unique connected
complex Lie group whose Lie algebra is the complexification of that
of $\fs$, the Lie algebra of $S$. Then
the map $C$ from $S\times\fs$ to $S^\CC$ given by 
\beq
\lb{CartanMapFirstCaseEq}
(s,X)\mapsto C(s,X):=s\exp_I\JJJ X
\eeq 
is a diffeomorphism.
\ethm

The following result can be thought of as an extension of Theorem \ref{CartanThm}
to the setting of a compact (but not necessary semisimple) Lie group.
We state a result in a form that will be most useful for us in applications,
albeit it not being quite optimal, perhaps.

\bprop
\label{PartialCartanProp}
Let $K$ be a 
compact connected subgroup of a connected complex
Lie group $(G,\JJJ)$.
Denote by $\fk$ and $\fg$ their Lie algebras. Suppose that
\beq
\label{kLieEq}
\baeq
\fk &=\fa\op\tfk,\cr
\fg &=\fa\op\tfk\op \JJJ\tfk,
\eaeq
\eeq
where $\fa$ is a complex Lie subalgebra of $\fg$ contained in the  center $\fz(\fk)$ of $\fk$,
and $\tfk$ is a Lie subalgebra of $\fk$.
Then the map 
$C:K \times \tfk \to G$ given by 
$C(k,X)=k\exp_{I}\JJJ X$ is surjective.
\eprop

\bpf
We start with several simple claims.
\bclaim
\label{fzfkLemma}
$
\fz(\fk)=\fa\op\fz(\tfk)$.
\eclaim

\bpf
The inclusion $\fz(\fk)\supset\fa\op\fz(\tfk)$ follows from \eqref{kLieEq}.
For the converse, let $X\in \fz(\fk)$. Since $X\in\fk$, 
by \eqref{kLieEq},
$X=X_1+X_2$ for $X_1\in \fa, X_2\in \tfk$. Since $X_2\in \fz(\fk)$,
and $\tfk\cap \fz(\fk)\subset\fz(\tfk)$, we are done.
\epf

\bclaim
\lb{tfkLemma}
(i)
$[\fk,\fk]=[\tfk,\tfk]$.
(ii)
$\tfk=\fz(\tfk)\op[\tfk,\tfk]$.

\eclaim

\bpf
(i) This follows from \eqref{kLieEq} since $[\fa,\fk]=0$. 

\noindent
(ii) Since $K$ is compact \cite[Proposition 6.6 (ii), p. 132]{Helg} gives
\beq
\lb{HelgEq}
\fk=\fz(\fk)\op [\fk,\fk].
\eeq 
From (i) and Claim \ref{fzfkLemma}, 
$$
\fk=\fz(\fk)\op[\tfk,\tfk]=\fa\op\fz(\tfk)\op[\tfk,\tfk].
$$
The conclusion now follows from \eqref{kLieEq}.
\epf

\bclaim
\lb{centerLemma}
$\fz(\fg)=\fa\op\fz(\tfk)\op \JJJ\fz(\tfk)$.
\eclaim

\bpf
First, $\fa\subset \fz(\fk)\subset\fz(\fg)$ by \eqref{kLieEq}.
By Claim \ref{fzfkLemma}, $\fz(\tfk)\subset\fz(\fk)\subset\fz(\fg)$.
Since $\fz(\fg)$ is complex, we obtain 
$\fz(\fg)\supset\fa\op\fz(\tfk)\op \JJJ\fz(\tfk)$.
Conversely, let $X\in\fz(\fg)$. By \eqref{kLieEq}, $X=X_1+X_2+X_3$, 
with $X_1\in\fa, X_2\in \tfk\cap \fz(\fg)=\fz(\tfk)$, 
and $X_3\in J\tfk\cap\fz(\fg)=\JJJ(\tfk\cap\fz(\fg))=J\fz(\tfk)$,
since $\fz(\fg)$ is complex.
\epf

Let $Z(K)$ and $Z(G)$ 
denote the connected closed Lie groups
whose Lie algebras are $\fz(\fk)$ and $\fz(\fg)$. 

\bclaim
\lb{Theta2Lemma}
The map 
$\Th_1:Z(K)\times\fz(\tfk)\ra Z(G)$ given by $(z,X)\mapsto z\exp_I\JJJ X$
is surjective. 
\eclaim

\bpf
Claims \ref{fzfkLemma} and \ref{centerLemma}
imply that $\dim Z(K)+\dim\tfk=\dim Z(G)$ and the differential of
$\Th_1$ at $(I,0)$ is invertible by \cite[Proposition 1.6, p. 104]{Helg}.
But, considering $\fz(\tfk)$ as an abelian Lie group with respect to
the additive structure, it follows that $\Th_1$ is a Lie group homomorphism,
thus it must be surjective as its image is a connected subgroup of the
same dimension as that of $Z(G)$.
\epf

We now conclude the proof of Proposition \ref{PartialCartanProp}.
Let $L$ denote the connected compact Lie subgroup of $K$ whose
Lie algebra is $[\tfk,\tfk]$ (since the Killing form
is negative on $[\tfk,\tfk]=[\fk,\fk]$, $L$ is indeed compact). 
By Claim \ref{tfkLemma}
and \cite[Proposition 6.6 (i), p. 132]{Helg}, $L$ is semisimple.
By Theorem \ref{CartanThm}, the map 
$\Th_2:L\times[\tfk,\tfk]\ra L^\CC$ given by
$$
\Th_2(l,X)=l\exp_I\JJJ X,
$$
is a diffeomorphism. 

\bclaim
\lb{MultMapLemma}
The multiplication maps $Z(K)\times L\ra K$, $Z(G)\times L^\CC\ra G$ are surjective.
\eclaim

\bpf
By Claim \ref{tfkLemma} (i) and \eqref{HelgEq},
the multiplication map $Z(K)\times L\ra K$ is a local isomorphism near $(I,I)$ by dimension count. The map is also a group homomorphism since elements of $Z(K)$
commute with elements of $K\supset L$. Thus, it is surjective.

The same argument works for the multiplication map $Z(G)\times L^\CC\ra G$. Here the dimension count is provided by Claim
\ref{tfkLemma} (ii), Claim \ref{centerLemma} and \eqref{kLieEq}, as together they give
$\fg = \fz(\fg)\op[\tfk,\tfk]\op \JJJ[\tfk,\tfk]$.
\epf

Given $k \in K$ and $X \in \fk$, observe that
$$
C(k,X)=zl\exp_{I}\JJJ X,
$$
where $z\in Z(K)$ and $l\in L$ are such that $k=zl$
(these exist by Claim \ref{MultMapLemma}).
Now let $X_1$ and $X_2$  
are the unique elements such that
$X_1\in \fz(\tfk)$, $X_2\in [\tfk,\tfk]$, and $X=X_1+X_2$,
given by Claim \ref{tfkLemma} (ii).
Since $\exp_I\JJJ X_1\in Z(G)$, 
$$
C(k,X)=z\exp_{I}\JJJ X_1 l\exp_I\JJJ X_2
=\Th_1(z,X_1)\Th_2(l,X_2).
$$
By Claim \ref{Theta2Lemma} and the fact that $\Th_2$ is a diffeomorphism,
it follows that $C$ surjects onto $Z(G)\times L^\CC$. However,
the multiplication map $Z(G)\times L^\CC\ra G$ is surjective by
Claim \ref{MultMapLemma}. Thus, $C$ is surjective, concluding the proof of Proposition \ref{PartialCartanProp}.
\epf

\subsection{Properness of the distance function on orbits }

Given data $(\mathcal R,d,F,G)$ satisfying \ref{a1}-\ref{a4}, the following result gives a criteria that implies a stronger version of condition  \ref{p6}. Though it may seem artificial at first glance, all the assumptions are verified naturally in the presence of canonical metrics.

\begin{proposition} 
\label{p6Prop}
Suppose   $K$ and $G$ satisfies the assumptions of Proposition \ref{PartialCartanProp}, $(\mathcal R,d,F,G)$ satisfies \ref{a1}-\ref{a4}, \ref{p4}, and $w \in \mathcal R$. We additionally assume the following:\\
\noindent (i) $K.w=w$. \\
\noindent (ii) For each $X\in\tfk$, $t\mapsto\exp_It\JJJ X.w$ 
is a $d$-geodesic whose speed depends continuously on $X$.\\
\noindent (iii) $G\times G \ni (f,g)\mapsto d(f.u,g.v)$
is  a continuous map for every $u,v\in\mathcal R$. \\
\noindent Then, for any $u,v\in\overline{\mathcal R}$ there exists $g\in G$
such that $d(u,g.v)=d_G(Gu,Gv)$.
\end{proposition}

\bpf
Let $u,v\in \mathcal R$. By \ref{p4} and Proposition
\ref{PartialCartanProp},
\beq
\lb{InfdGEq}
d_G(Gu,Gv)=\inf_{g\in G}d(u,g.v)=\inf_{k\in K,X\in\tfk}d(u,C(k,X).v).
\eeq
By \ref{p4},
$$
\baeq
d(u,C(k,tX).v)
&\ge
d(w,C(k,tX).w)-d(w,u)-d(C(k,tX).w,C(k,tX).v)
\cr
&=c_Xt-d(w,u)-d(w,v),
\eaeq
$$
since using \ref{p4}, (i) and (ii) we have
$$
d(w,C(k,tX).w)=d(k^{-1}.w,\exp_It\JJJ X.w)=d(w,\exp_It\JJJ X.w)=c_Xt,
$$
with $c_X$ depending continuously on $X\in\tfk$.
Since $\tfk$ is finite-dimensional it follows 
that $(k,X)\mapsto d(u,C(k,X).v)$ is proper. Hence the infimum in
\eqref{InfdGEq} is attained, because by (iii) 
$(k,X)\mapsto d(u,C(k,X).v)$ is continuous. This finishes the proof for $u,v \in \mathcal R$.

Finally, when $u,v\in\overline{\mathcal R}$, 
let $\{u_j\},\{v_j\}\subset \mathcal R$ denote sequences
that $d$-converge to $u$ and $v$. Then,
$$
\baeq
|d(u,C(k,X).v)-d(u_j,C(k,X).v_j)|
&\le 
d(u,u_j) + |d(u,C(k,X).v)-d(u,C(k,X).v_j)|
\cr
&\le
d(u,u_j) + d(C(k,X).v, C(k,X).v_j)
\cr
&=
d(u,u_j) + d(v, v_j),
\eaeq
$$
hence the continuous proper maps $(k,X) \to d(u_j,C(k,X).v_j)$ 
converge uniformly to $(k,X) \to d(u,C(k,X).v)$, making this latter map also continuous and proper. This gives $d_G(Gu,Gv)=d(u,g.v)$ for some $g \in G$, finishing the proof.
\epf
\subsection{Automorphism groups of canonical \K manifolds}

In this section we recall some classical theorems about the automorphism group of
a K\"ahler manifold $(M,\JJJ,\o)$ when the metric $\o$ is canonical,
following Gauduchon \cite{ga} to which we refer for more details. These results will be very helpful once we try to verify the conditions of Proposition \ref{p6Prop} in concrete situations.

Let $g(\,\cdot\,,\,\cdot\,)=\o(\,\cdot\,,\JJJ\,\cdot\,)$ denote the Riemannian
metric associated to $(M,\JJJ,\o)$.
Denote by $\Isom(M,g)_0$ the identity component of the isometry group of 
$(M,g)$. Since $M$ is compact so is $\Isom(M,g)_0$ \cite[Proposition 29.4]{Post}.
Denote by $\isom(M,g)$ the Lie algebra of $\Isom(M,g)_0$.
Consider the Lie subalgebra of $\autMJ$ of harmonic fields, 
\beq
\lb{faEq}
\mathfrak{a}:=\{X\in\autMJ\,:\, g(X,\,\cdot\,) \h{ is a $g$-harmonic 1-form}\},
\eeq
and the Lie subalgebra of $\isom(M,g)$ of Hamiltonian fields,
\beq
\lb{fhEq}
\mathfrak{h}:=\{X\in\isom(M,g)\,:\, \iota_X\o \h{ is an exact 1-form}\}.
\eeq

The following theorem is due to Matsushima and Lichnerowitz 
\cite[Theorem 3.6.1]{ga}.

\bprop
\lb{VectorFieldDecompProp}
Let $(M,\JJJ,\o,g)$ be as above. Suppose $g$ has constant scalar curvature. Then,
\begin{equation}
\label{ksplit}
\isom(M,g)=\mathfrak{a} \oplus \mathfrak{h},
\end{equation}
\begin{equation}\label{gsplit2}
\autMJ=\mathfrak{a} \oplus \mathfrak{h} \oplus \JJJ\mathfrak{h}.
\end{equation}
\eprop
In particular, Proposition \ref{VectorFieldDecompProp} implies
that $\Isom(M,g)_0\subset\AutMJz$.

Analogues of this result have been established in several settings. 
First, consider the 
Lie subalgebra \cite[Lemma A.2]{TZ}
\beq
\baeq
\lb{autXEq}
\aut^X(M,\JJJ)
&:=
\{Y\in\autMJ\,:\, [X,Y]=0\},
\eaeq
\eeq
and denote the associated connected complex Lie group
by 
\beq
\lb{AutMJXEq}
\Aut^X(M,\JJJ)_0\subset \AutMJz.
\eeq
Tian--Zhu proved the following result \cite[Appendix A]{TZ}.

\bprop
\lb{VectorFieldDecompSolitonProp}
Let $(M,\JJJ,\o,g)$ be as above, and let $X\in\autMJ$. Suppose 
$g$ is a \KR soliton. Then,
\begin{equation}\label{gsplit1}
\aut^X(M,\JJJ)=\isom(M,g)\oplus \JJJ\,\isom(M,g).
\end{equation}

\eprop

Next, let $D\subset M$ be a smooth divisor and consider
\beq
\baeq
\lb{autDEq}
\aut(M,D,\JJJ)
&:=
\{Y\in\autMJ\,:\, (1-\be)Y \h{ is tangent to $D$}\},
\eaeq
\eeq
and denote the associated connected complex  Lie group
by 
\beq \label{AutMDEq}
\Aut(M,D,\JJJ)_0\subset \AutMJz.
\eeq
Cheltsov--Rubinstein proved the following \cite[Theorem 1.12]{CR}.

\bprop
\lb{VectorFieldDecompEdgeProp}
Let $(M,D,\JJJ,\o,g)$ be as above and let $\be\in(0,1]$. Suppose 
$g$ is a \KE edge metric. Then,
\begin{equation}\label{gsplit}
\aut(M,D,\JJJ)=\isom(M,g)\oplus \JJJ\,\isom(M,g).
\end{equation}
\eprop

The following result is classical, and we only state its \KE version, whose proof we sketch.
\bthm
\lb{MatsIwasawaThm}
Let $(M,\JJJ,\o,g)$ be \KEno.
Then any maximally compact subgroup of $\AutMJz$
is conjugate to $\Isom(M,g)_0$.
\ethm

\begin{proof}By a Theorem of Iwasawa--Malcev \cite[Theorem 32.5]{st}, if $G$
is a connected Lie group then its maximal compact
subgroup must be connected and any two maximal compact subgroups are conjugate. But then by Proposition
\ref{VectorFieldDecompEdgeProp} ($\beta=1$) $\Isom(M,g)_0$ has to be a maximal compact
subgroup of $\AutMJz$.
\end{proof}

\section{\KE metrics}
\lb{KESec}

In this section we prove two results about existence of \KE metrics. Recall that $d_{1,G}$ and $J_G$ are defined in \ref{a4}
and \eqref{JGEq}, and $F^1$ and $E^1$
are defined in \eqref{FbetaEq} and \eqref{DingTianEq}. Our first result gives characterizes K\"ahler classes admitting \KE metrics. 
This theorem will be generalized in two different directions in 
the next two sections.

\begin{theorem} 
\label{KEexistenceThm}
Suppose $(M,\JJJ,\o$) is Fano. Set $G:=\Aut(M,\JJJ)_0$ and 
let $F\in\{F^1,E^1\}$.
The following are equivalent:

\smallskip
\noindent (i) 
There exists a K\"ahler--Einstein metric in $\mathcal H$.

\noindent (ii) 
$F$ is $G$-invariant and for some $C,D >0$,
$$
F(u) \geq C d_{1,G}(G0,Gu) - D, \qq u\in\H_0.
$$

\noindent (iii) 
$F$ is $G$-invariant and for some $C,D >0$,
$$
F(u) \geq C J_G(Gu) - D, \qq u\in\H_0.
$$
\end{theorem}

The purpose of our second theorem is to indicate what modifications are 
necessary in Tian's original conjecture 
(Conjecture \ref{TianConj} (ii)). 
Let $K\subset \AutMJz$ be a compact Lie subgroup. Denote,
$$
\H^K:=\{\eta\in\H\,:\, g.\eta=\eta \h{\ for any $g\in K$}\}.
$$
The isomorphic space of potentials, denoted $\H_0^K$, 
is defined after \eqref{HKEq}. 

\begin{theorem}
\label{KEGexistenceThm}
Suppose $(M,\JJJ,\o)$ 
is Fano and that $K$ is a maximal compact subgroup of $\AutMJz$.  Assume that 
$\o\in\H^K$.
Finally, let $F\in\{E^1,F^1\}$.
The following are equivalent:

\smallskip
\noindent (i) 
There exists a K\"ahler--Einstein metric in $\mathcal H^K$
and $\AutMJz$ has finite center.

\noindent (ii) For some $C,D >0$ and all $u \in \mathcal H^K_{0}$,
$$
F(u) \geq C d_1(0,u) - D.
$$

\noindent (iii) 
For some $C,D >0$ and all $u \in \mathcal H^K_{0}$,
$$
F(u) \geq C J(u) - D.
$$
\end{theorem}

\bremark
\lb{Dingd1Remark}
In Theorems \ref{KEexistenceThm}
and \ref{KEGexistenceThm} when $F=F^1$ one {\it cannot} replace $d_1$ by
the Mabuchi metric $d_2$. 
Indeed, according to \eqref{FbetaEq}, Jensen's inequality, and
Proposition \ref{Jproperness}, for $\vp\in\H_0$,
\beq
\label{FbetaAM0Eq}
F^1(\vp)
=
-\log \V\int_M e^{f_{\o}-\vp}{\o}^n
\le 
\V\int_M (\vp-f_\o)\on
\le
C(\sup_M\vp+1)\le C(d_1(0,\vp)+1). 
\eeq
If $F^1$ were $d_2$-proper it would follow that
$d_2(0,\vp)\le C' (d_1(0,\vp)+1)$. However, this is impossible.
E.g., when $M$ is toric, each of the $d_p$ metrics are equivalent (via the Legendre transform) to the $L^p$ metric on the space of convex functions on the Delzant polytope $P$ of $M$ \cite[Proposition 4.5]{guedj}, and one can construct $d_2$-unbounded sequences contained in a $d_1$-unit ball. More generally, for arbitrary $M$, the results of \cite{da4} can be readily used to construct such sequences. 
Finally, properness of the Calabi metric 
is not the correct notion either, since
$J$ is unbounded on $\H$ (Proposition \ref{Jproperness}) while Calabi's metric has finite diameter \cite{Calabi54}.

In connection with \cite[Conjecture 6.1]{c1}, it would be interesting to see if similar facts also hold the for the K-energy $E^1$.
\eremark

\begin{remark}
It would be interesting to extend Theorem \ref{KEGexistenceThm}
to the settings of \KR solitons or \KE edge metrics. 
Also, it is possible to modify the proof of Theorem 
\ref{KEexistenceThm} to other settings, e.g., 
Sasaki--Einstein metrics, twisted
KE metrics \cite{zz}, or multiplier Hermitian structures
\cite{Mab2003}.
For brevity, we do not pursue these here and leave this
and related extensions to the reader.
\end{remark}

\subsection{Proof of Theorem \ref{KEexistenceThm}}
\lb{FirstKEThmSubSec}

The equivalence of (ii) and (iii) is the content of Lemma
\ref{JGPropernessLemma}.

For the equivalene between (i) and (ii) 
we wish to apply Theorem \ref{ExistencePrinc} to the data
$$
\mathcal R=\H_0, \q d=d_1, \q 
F\in\{E^1,F^1\}, \q G:=\Aut_0(M,\JJJ).
$$
First, we go over Notation \ref{MainNot}.
First, in \ref{a1}, 
$\overline{\mathcal R}=\E_1\cap \h{\rm AM}^{-1}(0)$ by Theorem \ref{d1CompletionThm}
and Lemma \ref{E1capH0Lemma}.
Observe that \ref{a2} holds by Propositions \ref{FbetaExt} and \ref{EbetaExt} (with $\be=1$).
In \ref{a3}, the minimizers of $F$ are denoted by $\calM$.
Finally, \ref{a4} holds since $G\subset \AutMJz$ implies
that if $g\in G$ and $\eta\in\H$ then 
$g.\eta$ is both \K and cohomologous to $\eta$, i.e., $g.\eta\in\H$.
Thus, it remains to verify Hypothesis \ref{MainHyp}.

\begin{enumerate}[label = (P\arabic*)]
  \item
This is due to
Berndtsson \cite[Theorem 1.1]{brn} for $F^1$ and to
Berman--Berndtsson \cite[Theorem 1.1]{bb} for $E^1$.
  \item 
For $E^1$ this is Proposition \ref{Ebetacompact} with $\beta=1$.
For $F^1$, this follows from Proposition \ref{Fbetacompact} with $\beta=1$. 
  \item 
This is Theorem \ref{BRM1Thm} with $\beta=1$.  
  \item 
This is Lemma \ref{dpIsomLemma}.
  \item 
This follows from \ref{p3} and the Bando--Mabuchi uniqueness theorem 
\cite[Theorem A (ii)]{BM}. 
  \item Suppose $u \in \mathcal H_0$ is a K\"ahler-Einstein metric. We wish to apply Proposition \ref{p6Prop}  
with $K=\Isom(M,g_{u})_0$ and $G=\AutMJz$.
There are several points to check.
First, we verify the assumptions of Proposition \ref{PartialCartanProp} 
(used in Proposition \ref{p6Prop}):
(i) $K$ is a compact connected subgroup of $G$ by
Proposition \ref{MatsIwasawaThm}; 
(ii) According to Proposition
\ref{VectorFieldDecompEdgeProp} (with $\be=1$), Equation \eqref{kLieEq} holds
with $\fa=0$ and $\tfk=\isom(M,g_{\o_u})$.
Second, we verify the assumptions of Proposition \ref{p6Prop}:\newline
(i) $K.u=u$ by definition; \newline
(ii) 
for each $X\in\tfk$, $t\mapsto\exp_It\JJJ X.u$ 
is a $d_1$-geodesic. This is classical 
since 
by \eqref{XDecompEq} and the fact that 
$X \in \isom(M,{g_{\o_u}})$ it 
follows that  
\beq
\lb{JXGradEq}
\JJJ X=\nabla\psi^{\JJJsml X}_{\o_u}
\eeq
is a gradient (with respect to $g_{\o_u}$) vector field
\cite[Theorem 3.5]{Mabuchi87}. 
Indeed,
first remark that 
Set $\o_{\vp(t)}:=\o(t)=\exp_It\JJJ X.\o_u$. 
Thus, 
\beq
\lb{dototEq}
\dot \o(t)=
\frac d{dt} 
\exp_It\JJJ X.\o_u = 
\calL_{\JJJsml X}\o\circ \exp_It\JJJ X
=\i\ddbar\psi^{\JJJsml X}_{\o_u}\circ\exp_It\JJJ X,
\eeq
and 
$$
\ddot \o(t)=\i\ddbar|\nabla\psi^{\JJJsml X}_{\o_u}|^2\circ\exp_It\JJJ X,
$$
i.e., $\ddot \vp(t)-|\nabla\dot\vp(t)|^2_{\o_{\vp(t)}}=0$,
which by an observation of Semmes and Donaldson 
\cite{Semmes,do} means that $\vp(t)$ solves
\eqref{MabuchiEq}. 
Thus, Theorem \ref{d1CompletionThm} implies
$t\mapsto \vp(t)$ is a $d_1$-geodesic.
The speed of this geodesic depends continuously on $X$ by 
\eqref{dototEq} and \eqref{distgeod}; \newline
(iii) $G\times G \ni (f,g)\mapsto d_1(f.u,g.v)=d_1(u,f^{-1}\circ g.v)$
by Lemma \ref{dpIsomLemma}, and this is a continuous map in $G\times G$
whenever $u,v\in\mathcal \H_0$ are fixed. Indeed, if $h_k\in G$
 converges to $h\in G$ then $h_k.\o_v$  
converges smoothly to $h.\o_v$ and using the Green kernel
of $\o$ we see that also $h_k.v$ converges smoothly to $h.v$.
Thus, $d_1(u,h_k.v)$  converges to $d_1(u,h.v)$ by \eqref{d1CharFormula}.
  \item 
Both $F^1$ and $E^1$ are the path-integrals of $G$-invariant
closed $1$-forms on $\calH_\o$.

\end{enumerate}

\subsection{Proof of Theorem \ref{KEGexistenceThm}}

By Proposition \ref{Jproperness}, it suffices 
to verify the equivalence between (i) and (ii). 
We apply Theorem \ref{ExistencePrinc}, but this time only in 
the direction (i) $\Rightarrow$ (ii).
To do so, we set
$$
\mathcal R=\H_0^K,\q d=d_1,\q 
F\in\{E^1,F^1\}, \q G=\{ I \}.
$$
By Lemma \ref{H0KCompletionLemma}, $\overline{\calR}=\E^K_1\cap \h{\rm AM}^{-1}(0)$. All the properties \ref{a1}-\ref{a4}, \ref{p1}-\ref{p7} are inherited from Therem \ref{KEexistenceThm}, with the exception of \ref{p3} and \ref{p5}. We verify these then and then Theorem \ref{ExistencePrinc} will automatically yield (ii).

\bclaim \label{p3Claim}
Assume that (i) holds. Property \ref{p3} holds.
\eclaim
\bpf
As $\mathcal H^K_0$ contains a \KE metric $u$, Theorem \ref{BRM1Thm} ($\beta=1$) gives that $u$ minimizes $F$ globally on $\mathcal E_1$, so in particular also on $\mathcal E^K_1$, giving $u \in \mathcal M$. If $v \in \mathcal M$ arbitrary, then $F(u)=F(v)$, hence another application of Theorem \ref{BRM1Thm} gives that $v$ is also smooth K\"ahler-Einstein, concluding that $\mathcal M \subset \mathcal H^K_0$. 
\epf

As we chose the group $G$ to be trivial, to verify \ref{p5}, we have to show that $\mathcal M$ is a singleton. Before proving this we need to understand properties of the group $K$. 

\bclaim
\lb{TrivialGroupClaim}
Suppose $(M,\JJJ,\o,g_{\o_u})$ is \KE with $u \in \mathcal H^K_0$. Then $K=\Isom(M,g_{\o_u})_0$.
\eclaim

\bpf By Theorem \ref{MatsIwasawaThm} $\Isom(M,g_{\o_u})_0$ is a maximal compact subgroup of 
$\AutMJz$. By assumption $K \subset \AutMJz$ is also maximal and trivially $K \subset \Isom(M,g_{\o_u})_0$, hence in fact $K = \Isom(M,g_{\o_u})_0$. 
\epf
Let $L$ be a group. Recall that the centralizer and normalizer of a subgroup $H$ are defined as follows:
$$
N_H(L):=\{g\in L\,:\, ghg^{-1}\in H, \ \forall h \in H\}.
$$
$$
C_H(L):=\{g\in L\,:\, ghg^{-1}=h, \ \forall h \in H\}\subset N_H(L).
$$
Note that $C_L=C_L(L)$ is just the center of $L$. The following result is due to Hazod et al. \cite[Theorem A]{hhswz} and we will make us of it shortly.

\begin{theorem} 
\lb{HHSWZThm}Suppose $H$ be a compact subgroup of a connected Lie group $L$.  
Then the group $N_H(L)/(H C_H(L))$ is a finite.
\end{theorem}

\begin{lemma} 
\lb{NormalizerLemma}
Suppose $(M,\JJJ,\o,g_{\o_u})$ is \KE and that $\AutMJz$ has finite center. Then, 
$N_{\Isom(M,g_{\o_u})_0}(\AutMJz)=\Isom(M,g_{\o_u})_0$.
\end{lemma}

\begin{proof} 
We claim that
\beq
\lb{equalityCentralizersEq}
C_{\Aut(M,\JJJsml)_0}(\AutMJz)
=
C_{\Isom(M,g_{\o_u})_0}(\AutMJz).
\eeq
One inclusion is by definition; for the converse
suppose that $h\in C_{\Isom(M,g_{\o_u})_0}(\AutMJz)$. The
map $\AutMJz\ni g\mapsto C_h(g):=hgh^{-1}\in \AutMJz$ is biholomorphic. Thus, $dC_h(\JJJ X)=\JJJ dC_h(X)$. Since
$C_h(g)=g$ whenever $g\in \Isom(M,g_{\o_u})_0$, it follows
that $dC_h(X)=X$ for each $X\in\isom(M,g_{\o_u})$. 
It follows from Proposition \ref{VectorFieldDecompProp}
that in fact $dC_h=\Id$ identically. Since $\AutMJz$
is connected (i.e., the union of all of its 1-parameter subgroups
passing through the identity), it follows that $C_h$ is
the identity map. Thus, $h\in C_{\Aut(M,\JJJsml)_0}(\AutMJz)$.

Hence, By Theorem \ref{HHSWZThm},
$$
N_{\Isom(M,g_{\o_u})_0}(\AutMJz)/\Isom(M,g_{\o_u})_0
$$ 
is finite, 
implying that $N_{\Isom(M,g_{\o_u})_0}(\AutMJz)$ is compact. 
By Theorem \ref{MatsIwasawaThm}, 
it follows that $N_{\Isom(M,g_{\o_u})_0}(\AutMJz)=\Isom(M,g_{\o_u})_0$.
\end{proof}

\bclaim \label{SingletonClaim}  Assume that (i) holds. Then  $\calM$ is a singleton, hence \ref{p5} holds.
\eclaim
\begin{proof} Suppose $u,v \in \mathcal M$. By Claims \ref{p3Claim} and \ref{TrivialGroupClaim} we have $\Isom(M,g_{\o_v})_0=\Isom(M,g_{\o_u})_0=K$. 
Thus,
$$
f^{-1} K f \subset \Isom(M,f^\star g_{\o_u})_0
=
\Isom(M,g_{\o_v})_0=K.
$$ 
Thus, $f \in N_K(\AutMJz)$. By Lemma \ref{NormalizerLemma}, $N_K(Aut_0(M,\JJJ))=K$, so $u=f. u = v$.
\end{proof}

We now turn to proving the direction  (ii) $\Rightarrow$ (i).

\bclaim
\lb{ContMethClaim}
If (ii) holds,
$\mathcal H^K_{0}$ contains a \KE potential.
\eclaim

\bpf
As remarked earlier, (iii) holds by Proposition \ref{Jproperness}.
The classical continuity method with a $K$-invariant reference metric
(observe $\H^K$ is nonempty by averaging an arbitrary element of $\H$ with
respect to the Haar measure of $K$) produces a \KE potential 
when $F=F^1$
\cite[Proposition]{Tian97}
or $F=E^1$ \cite[pp. 2651--2653]{RJFA}. 
Since solutions $\vp(t)$ to $\o_{\vp(t)}^n=\on e^{f_\o-t\vp(t)}$ are unique for each $t\in(0,1)$
\cite[Theorem 7.1]{brn}, it follows that the \KE potential must be $K$-invariant.
\epf

For the remainder of the proof we fix a \KE potential $u \in \mathcal H^K_0$ provided by the last claim. By Claim \ref{TrivialGroupClaim} we have $K = \Isom_0(M,g_{\o_u})$. It remains to show that the center of $\AutMJz$ is finite. For this we take a closer look at the normalizer of $K$:

\bclaim
\lb{InvCptClaim}
Suppose (ii) holds.
Then
$N_K(\AutMJz)$ is compact.
\eclaim
\bpf It is trivial to verify that $N_K(\AutMJz)$ acts on $\mathcal H^K_0$, hence  $g.u \in \mathcal H^K_0$ for any $g \in N_K(\AutMJz)$. As $g.u$ is \KE it follows that $F(u)=F(g.u)$. 

By (ii), 
$d_1(u,g.u)$ is uniformly bounded for $g \in N_K(\AutMJz)$. 
Let
$C$ be the map given by \eqref{CartanMapFirstCaseEq} 
(recall Proposition \ref{VectorFieldDecompEdgeProp})
and consider
$$
S:=C^{-1}(N_K(\AutMJz)).
$$
Since $C$ is continuous, $S$ is closed in  $K \times \isom(M,g_{\o_u})$. To show that $N_K(\AutMJz)$ is compact we only need to show that $S$ is bounded. For any  $(k, X) \in S$ we can write
\begin{flalign*}
d_1(u,C(k,\JJJ X).u) =  d_1(u,k\exp_{I}\JJJ X.u)=d_1(u,\exp_{I}\JJJ X.u)\ge \rho |X|,
\end{flalign*}
where, using Theorem \ref{d1Thm},
$$
\rho:= \inf_{Y \in \isomsml(M,g_{\o_u}), |Y|=1}
d_1(u, \exp_{I}\JJJ Y.u)>0,
$$
since $[0,\infty) \ni t \to \exp_{I}t\JJJ X.u \in \mathcal H_{0}$ 
is a $d_1$-geodesic ray initiating from $u$ by 
the proof of \ref{p6} in \S\ref{FirstKEThmSubSec}.
Thus, $|X|$ is uniformly bounded giving that $S$ is a bounded set.
\epf

As $K=\Isom(M,g_{\o_u})_0$ is maximally compact in $\AutMJz$, 
$$
N_K(\AutMJz)=K.
$$
Thus, by \eqref{equalityCentralizersEq} (which uses only Proposition \ref{VectorFieldDecompProp}), 
\beq
\lb{LastCInclEq}
C_{\Aut(M,\JJJsml)_0}
=
C_K(\AutMJz)
\subset
N_K(\AutMJz)=K.
\eeq
But the Lie algebra of $C_{\Aut(M,\JJJsml)_0}$
must be trivial, since it is complex and is a Lie subalgebra of 
$\isom(M,g_{\o_u})$. Indeed,  by  Proposition \ref{VectorFieldDecompProp},
the only complex Lie subalgebra of $\isom(M,g_{\o_u})$ is the trivial one.
Since $C_{\Aut(M,\JJJsml)_0}$ is compact by \eqref{LastCInclEq}, it must
be finite. This concludes the proof of Theorem \ref{KEGexistenceThm}.

\section{\KR solitons}
\lb{KRSSec}

We now state our main result concerning existence of \KR solitons.
Recall the definition of $\aut^X(M,\JJJ)$ and $\Aut^X(M,\JJJ)$
from \eqref{AutMJXEq}. In this section we set
\beq
\lb{GKRSDef}
G:=\Aut^X(M,\JJJ).
\eeq

\begin{theorem} 
\label{solitonKEexistenceThm}
Let $(M,\JJJ,\o)$ be a Fano manifold with $[\o]=c_1(X)$,
let $X\in\autMJ$, let $T$ be the group determined by $X$
given in \eqref{AssumpXEq}, and let $F\in\{F^X,E^X\}$.
The following are equivalent:

\noindent 
(i) 
There exists a \KR soliton associated to $X$ in $\mathcal H_{0}^T$.

\noindent 
(ii) 
$F$ is $G$-invariant and for some $C,D >0$,
$$
F(u) \geq C d_{1,G}(G0,Gu) - D, \qq u\in\H^T_0.
$$
\noindent 
(iii) 
$F$ is $G$-invariant and for some $C,D >0$,
$$
F(u) \geq C J_G(G0,Gu) - D, \qq u\in\H^T_0.
$$

\end{theorem}

\begin{proof}
By Lemma \ref{JGPropernessLemma}, it suffices 
to verify the equivalence between (i) and (ii). 
We apply Theorem \ref{ExistencePrinc}, but this time only 
for $F=F^X$. Recall $G$ is defined in \eqref{GKRSDef}. 
For the remaining data in Notation \ref{MainNot} we set
$$
\mathcal R=\H_0^T,\q d=d_1.
$$
Thus, in \ref{a1}, 
$\overline{\calR}=\E^T_1\cap \h{\rm AM}^{-1}(0)$,
by Lemma \ref{H0KCompletionLemma}.
Observe that \ref{a2} holds by Proposition \ref{FModifExt}.
In \ref{a3}, the minimizers of $F^X$ are denoted by $\calM$.
Finally, \ref{a4} holds by Lemma \ref{ActionExtension}.
Thus, it remains to verify Hypothesis \ref{MainHyp}.

\begin{enumerate}[label = (P\arabic*)]
  \item
This is due to \cite[Theorem 1.1, Proposition 10.4]{brn} .
  \item 
This is Proposition \ref{FModifCompact}. 
  \item 
This is Theorem \ref{BWNRThm}.  
  \item 
This is Lemma \ref{dpIsomLemma}.
  \item 
This follows from \ref{p3} and the Tian--Zhu uniqueness theorem 
\cite{TZ}.
  \item Suppose $u \in \mathcal M$ is a \KR solition. We may apply Proposition \ref{p6Prop}  
with $K=\Isom(M,g_{\o_u})_0$ and $G$, since
the assumptions of Proposition \ref{PartialCartanProp} 
are verified by Proposition \ref{VectorFieldDecompSolitonProp}
and an argument identical to the proof of \ref{p6} in \S\ref{FirstKEThmSubSec}.
  \item 
$F^X$ is the path-integrals of $G$-invariant
closed $1$-forms on $\calH^X_\o$.

\end{enumerate}

This concludes the proof of Theorem \ref{solitonKEexistenceThm} for
$F=F^X$.

We now let $F=E^X$. By \eqref{EmodifFmodifIneq} and the previous paragraph
it suffices to verify the direction (ii) $\Rightarrow$ (i). 
This is done in \cite{TZ},\cite[Theorem 1.6]{bwn}. 
\end{proof}

\section{K\"ahler--Einstein edge metrics.} 
\lb{KEESec} 

We now state our main result concerning existence of \KE edge metrics. Recalling \eqref{AutMDEq}, in this section we set
\beq
\lb{GKEEDefEq}
G:=\Aut(M,D,\JJJ)_0.
\eeq

\begin{theorem} 
\label{KEEexistenceThm}
Suppose $(M,\JJJ,\o)$ is compact K\"ahler, $D\subset M$, and $\be\in(0,1)$ 
satisfy \eqref{ConicCohomCond}. 
Let $F\in\{F^\be,E^\be\}$.
The following are equivalent:

\smallskip
\noindent (i) 
There exists a K\"ahler--Einstein edge metric in $\mathcal H^\be$.

\noindent (ii) 
$F$ is $G$-invariant and for some $C,D >0$,
$$
F(u) \geq C d_{1,G}(G0,Gu) - D, \qq u\in\H^\be_0.
$$

\noindent (iii) 
$F$ is $G$-invariant and for some $C,D >0$,
$$
F(u) \geq C J_G(Gu) - D, \qq u\in\H^\be_0.
$$
\end{theorem}

\bpf
By Lemma \ref{JGPropernessLemma}, it suffices 
to verify the equivalence between (i) and (ii). 
We apply Theorem \ref{ExistencePrinc}, but as in \S\ref{KRSSec} 
only for $F=F^\beta$. Recall $G$ is defined in \eqref{GKEEDefEq}. 
For the remaining data in Notation \ref{MainNot} we set
$$
\mathcal R=\H_0^\be,\q d=d_1.
$$
Thus, in \ref{a1}, 
$\overline{\calR}=\E_1\cap \h{\rm AM}^{-1}(0)$,
by Lemmas \ref{E1capH0Lemma} and \ref{HbetaDense}.
Observe that \ref{a2} holds by Proposition \ref{FbetaExt}.
In \ref{a3}, the minimizers of $F$ are denoted by $\calM$.
Finally, \ref{a4} holds by Definition \ref{KEEDef}. 
Thus, it remains to verify Hypothesis \ref{MainHyp}.

\begin{enumerate}[label = (P\arabic*)]
  \item
This is \cite[Theorem 6.4]{brn}.
  \item 
This is Proposition \ref{Fbetacompact}. 
  \item 
This is Theorem \ref{BRM1Thm}.  
  \item 
This is Lemma \ref{ActionExtension}.
  \item 
This follows from \ref{p3} and the Berndtsson's uniqueness theorem 
\cite[Theorem 6.4]{brn}.
  \item Let $u \in \mathcal M$ be a \KE edge metric. We may apply Proposition \ref{p6Prop}  
with $K=\Isom(M,D,g_{\o_u})_0$ and $G$, since
the assumptions of Proposition \ref{PartialCartanProp} 
are verified by Proposition \ref{VectorFieldDecompEdgeProp}
and the proof of \ref{p6} in \S\ref{FirstKEThmSubSec}.
  \item 
$F^\be$ is the path-integrals of $G$-invariant
closed $1$-forms on $\calH^\be_\o$.

\end{enumerate}

This concludes the proof of Theorem \ref{KEEexistenceThm} for
$F=F^\be$.

We now let $F=E^\be$. By \eqref{EmodifFmodifIneq} and the previous paragraph
it suffices to verify the direction (ii) $\Rightarrow$ (i).
For this, Theorem \ref{ExistencePrinc}, thanks to Remark \ref{MainThmRemark} (i)
and Proposition \ref{Ebetacompact}, implies that $\calM\neq\emptyset$. Finally, Theorem \ref{BRM1Thm} gives that $\calM\subset\calR$, as desired.
\epf

\section{Constant scalar curvature metrics} 
\lb{cscSec}

We now state our main result concerning existence of constant scalar curvature metrics. In this section we set
\beq
\lb{GKEDefEq}
G:=\Aut(M,\JJJ)_0.
\eeq

\begin{theorem} 
\lb{CscThm} Let $(M,\o)$ be a K\"ahler manifold and suppose that minimizers of the K-energy $E$ 
on $\E_1$ are smooth. Then the following are equivalent:
\smallskip
\noindent (i) 
There exists a constant scalar curvature metric in $\mathcal H$.

\noindent (ii) 
$E$ is $G$-invariant and for some $C,D >0$,
$$
E(u) \geq C d_{1,G}(G0,Gu) - D, \qq u\in\H_0.
$$

\noindent (iii) 
$E$ is $G$-invariant and for some $C,D >0$,
$$
F(u) \geq C J_G(Gu) - D, \qq u\in\H_0.
$$

\end{theorem}

\bpf
The proof is the same as that of Theorem \ref{KEexistenceThm}
except for the following points.
For the equivalence between (i) and (ii) 
we apply Theorem \ref{ExistencePrinc} to the data
$$
\mathcal R=\H_0, \q d=d_1,  \q F =E.
$$
Again, in \ref{a1}, 
$\overline{\mathcal R}=\E_1\cap \h{\rm AM}^{-1}(0)$.
Observe that condition \ref{a2} holds by Proposition \ref{EcscExtProp}.
Property \ref{p1} holds by \cite[Theorem 1.1]{bb}.
Property \ref{p2} holds by Proposition \ref{Ebetacompact}.
Property \ref{p3} holds by assumption. 
Property \ref{p5} holds by the work of Berman--Berndtsson
\cite[Theorem 1.3]{bb}. Suppose $u \in \mathcal H_0$ is  a constant scalar curvature metric. By the same argument as the one in Theorem \ref{KEexistenceThm}, due to Propositions \ref{p6Prop} and \ref{VectorFieldDecompProp}, property \ref{p6} holds 
with $K=\Isom(M,g_{\o_u})_0$ and $G=\AutMJz$.
\epf

\section*{Acknowledgments}
Research supported by BSF grant 2012236,
NSF grants DMS-1206284,1515703 and a Sloan Research Fellowship.

\def\bi{\bibitem}

\bigskip

\bigskip

{\sc University of Maryland} 

{\tt tdarvas@math.umd.edu, yanir@umd.edu}

\end{document}